\documentclass[11pt,twoside,reqno]{amsart}
\usepackage{amsmath, amsfonts, amsthm, amssymb,amscd, graphicx, amscd}
\usepackage{float,epsf}
\usepackage[english]{babel}
\usepackage{enumerate}
\usepackage{tikz}
\usepackage{hyperref}

\usepackage{srcltx}
\usepackage{geometry}
\usepackage{verbatim}
\usepackage{mathrsfs}
\newcommand{\R}{\mathbb{R}}

\renewcommand{\H}{\textcolor{red}{\mathscr{H}}}

 \newcommand{\supp}{\text{\rm supp}\,}

 \renewcommand{\supp}{\text{\rm supp}\,}

\newtheorem{theorem}{Theorem}[section]
\newtheorem{lemma}[theorem]{Lemma}

\newtheorem{definition}[theorem]{Definition}

\newtheorem{proposition}[theorem]{Proposition}
\newtheorem{corollary}[theorem]{Corollary}
\newtheorem{remark}[theorem]{Remark}
\newtheorem{example}[theorem]{Example}

\numberwithin{equation}{section}
\numberwithin{figure}{section}

\renewcommand{\div}{\mathrm{div}} %divergence
\newcommand{\res}{\mathop{\hbox{\vrule height 7pt width .5pt depth 0pt
\vrule height .5pt width 6pt depth 0pt}}\nolimits}

\def\intave#1{\int_{#1}\hbox{\llap{$\raise2.3pt\hbox{\vrule
height.9pt width7pt}\phantom{\scriptstyle{#1}}\mkern-2mu$}}}

\newtheorem{problem}{Problem}[section]

\everymath{\displaystyle}

\title[Heat insulation problem]{Compactness of $M$-uniform domains and optimal thermal insulation problems}
\author{Hengrong Du, \ Qinfeng Li, \ Changyou Wang}
\address{Department of Mathematics, Purdue University, West Lafayette, IN 47907, USA.}
\email{du155@purdue.edu}
\address{School of Mathematics, Hunan University, Changsha 410082, Hunan, P. R. China} \email{liqinfeng1989@gmail.com}
\address{Department of Mathematics, Purdue University, West Lafayette, IN 47907, USA.}
\email{wang2482@purdue.edu}

\begin{document}

\maketitle
\begin{abstract} In this paper, we will consider an optimal shape problem of heat insulation introduced by \cite{BBN1}.
We will establish the existence of optimal shapes in the class of $M$-uniform domains. We will also show that
balls are stable solutions of the optimal heat insulation problem.
\end{abstract}
\section{Introduction}
\subsection{Background}

In this paper, motivated by Bucur-Buttazzo-Nitsch in their papers \cite{BBN} and \cite{BBN1},  we consider the thermal insulation problem of designing the optimal shape $\Omega$ of $\mathbb{R}^n$ which represents a thermally conducting body, and determining the best distribution of a given amount of insulating material around $\Omega$; the thickness of the insulating material is assumed to be very small with respect to the size of $\Omega$ so the material density is assumed to be a nonnegative function defined on the boundary $\partial \Omega$. A rigorous approach is to consider a limit problem when the thickness of the insulating layer goes to zero and simultaneously the conductivity in the layer goes to zero. 

Mathematically, this amounts to consider the limit of the family of functionals, as $\epsilon\to 0$, 
\begin{equation}\label{ins_f1}
F_{\epsilon}(u, h, \Omega)=\frac{1}{2}\int_{\Omega}|\nabla u|^2\,dx+\frac{\epsilon}{2}\int_{\Sigma_\epsilon} |\nabla u|^2\,dx-\int_{\Omega} fu\,dx, 
\end{equation} over $u\in H^1_0(\Omega_{\epsilon})$, where $\Omega_\epsilon=\Omega \cup \Sigma_{\epsilon}$. 
Here $\Omega$ has a prescribed volume $V_0$, $\Sigma_\epsilon=\{\sigma+t\nu(\sigma): \sigma\in\partial\Omega, \, 0\le t \le \epsilon h(\sigma)\}$ is the thin layer of sickness $\epsilon h(\sigma)$ around $\partial \Omega$, and $h \in {\mathscr{H}}_m$, where 
$${\mathscr{H}}_m=\left\{h: \partial \Omega \rightarrow \mathbb{R} \, \mbox{is measurable}, \, h \ge 0, \int_{\partial \Omega} h d\sigma =m\right\}$$and $h$ denotes the distribution function of insulation material with fixed total amount $m>0$.

As in \cite{AB} and \cite{BBN}, in the framework of $\Gamma$-convergence 
passing to the limit $\epsilon \rightarrow 0$ in \eqref{ins_f1} we obtain the limit energy functional
\begin{align}
\label{lim1}
\mathcal{F}_m(u,h,\Omega)=\frac{1}{2}\int_{\Omega}|\nabla u|^2dx+\frac{1}{2}\int_{\partial \Omega} \frac{u^2}{h} d\sigma-\int_{\Omega} fudx.
\end{align}
By \cite{BB}, for any fixed $u$ and $\Omega$, if we minimize $F(u,h,\Omega)$ over $h \in {\mathscr{H}}_m$, 
then $F(u,h,\Omega)$ achieves its minimum when 
\begin{align}
\label{h}
\displaystyle h=m\frac{|u|}{\int_{\partial \Omega} |u| \,d\sigma}.
\end{align}
After substituting \eqref{h} for $h$ into \eqref{lim1}, we seek to minimize 
\begin{align}
\label{min1}
\mathcal{J}_m(u,\Omega):=\frac{1}{2}\int_{\Omega} |\nabla u|^2\,dx+\frac{1}{2m}\left(\int_{\partial \Omega} |u| \,d{\mathscr{H}}^{n-1}\right)^2-\int_{\Omega} fu \,dx
\end{align}
over all $ u \in H^1(\Omega)$, subject to the volume constraint $|\Omega|=V_0$. 

It was proved in \cite{BBN} that for every $f \in L^2(\Omega)$, if $\Omega$ is fixed, then the minimization of \eqref{min1} admits a unique solution $u_\Omega \in H^1(\Omega)$, and moreover if $\Omega=B_R$
and $f\equiv 1$, then 
$$u_{B_R}(x)=\frac{R^2-|x|^2}{2n}+\frac{m}{n^2\omega_nR^{n-2}},$$
where $\omega_n$ is the volume of unit ball in $\mathbb{R}^n$ 
and $B_R$ is the ball of radius $R$ centered at origin. 

Stationary solutions were also obtained in \cite{BBN}. More precisely,  
for a given smooth vector field  $\eta \in C_0^{\infty}(\mathbb{R}^n)$ with $\int_{\Omega} \div \eta {\, dx}=0$, 
let $F_t(x):=F(t,x)$ be the \textit{flow map} generated by the vector field $\eta$,
i.e., $F_t$ solves the ODE in $\R^n$:
\begin{equation}
  \left\{
    \begin{array}{l}
      \frac{d}{dt}F(t,x)=\eta(F(t,x))\\
      F_0(x)=x.
    \end{array}
    \right.
    \label{}
  \end{equation}
It was proved in \cite{BBN} that for $f \equiv 1$, $B_R$ is a stationary shape in the sense that 
$$\frac{d}{dt}\Big |_{t=0} \mathcal{J}_m(u_t, \Omega_t)=0,$$
where $u_t=u\circ F_t^{-1}$, $\Omega_t=F_t(B_R)$, and $|B_R|=V_0$.\\

Two open questions are asked by Bucur-Buttazzo-Nitsch in \cite{BBN1}.
\begin{problem}
\label{113ex}
Do the optimal shapes minimizing the energy functional \eqref{min1} exist?
\end{problem}
\begin{problem}
\label{1133}
Is it true that $B_R$ is a unique optimal shape when $f\equiv 1$?
\end{problem}

\subsection{Existence of minimizers over convex Domains}
There has been a developed scheme for the existence of a minimizer to the problem \eqref{min1} over convex domains contained within a container $B_{R}$ and $H^1$ function associated to such domains, due to the compactness properties of such domains, see \cite{Henrot}, \cite{BB} and the survey book \cite{nshap}. See also the paper \cite{LP} by Lin-Poon. Indeed, the existence of problem \eqref{min1} relies on the following properties for convex domains: If $\Omega \subset B_{R}$ is convex, $|\Omega|=V_0>0$ and $u \in H^1(\Omega)$, then
\begin{itemize}
\item [1.] (Uniform Poincar\'e inequality) There exists a universal constant $C>0$, independent of
$(u,\Omega)$, such that
\begin{align}
\label{unipoincare}
    \int_{\Omega} u^2 dx \le C\Big(\int_{\Omega} |\nabla u|^2 \,dx +\big(\int_{\partial \Omega} |u| \,d{\mathscr{H}}^{n-1}(x) \big)^2\Big).
\end{align} 
This guarantees the uniform $H^1$-bound of $u_i$ for any minimizing sequence
 $(u_i, \Omega_i)$ of $\mathcal{J}_m$.
\item [2.] (Uniform Sobolev extension property) 
There exists a universal constant $C>0$  independend of $\Omega$ such that for each 
$u \in H^1(\Omega)$, there exists $\tilde{u} \in H^1(\mathbb{R}^n)$ such that $\tilde{u}=u$ in $\Omega$,
and
\begin{align}
\label{uniextensionforconvex}
    \Vert \tilde{u} \Vert _{H^1(\mathbb{R}^n)} \le C \Vert u \Vert _{H^1(\Omega)}.
\end{align}

\item [3.] (Compactness of convex domains) If $\Omega_i$ is a sequence of convex sets in $B_{R}$ with $|\Omega_i|=V_0$, then there is a convex domain $\Omega$ such that $\Omega_i \rightarrow \Omega$ in $L^1$ ,
and $${\mathscr{H}}^{n-1}\res _{\partial \Omega_i} \rightarrow {\mathscr{H}}^{n-1}\res _{\partial \Omega}$$
as convergence of Radon measures.

\item [4.] (Lower semicontinuity of energy) From \eqref{unipoincare}, \eqref{uniextensionforconvex} and the compactness of convex domains in $B_{R}$, for any minimizing sequence of pairs $(u_i, \Omega_i)$ to \eqref{min1},  there are $\Omega$ and $u \in H^1(\Omega)$ such that up to a subsequence, 
$\Omega_i \rightarrow \Omega$ in $L^1$,
\begin{align}
 \int_\Omega |\nabla u|^2\,dx\le  \liminf_{i \rightarrow \infty} \int_{\Omega_i} |\nabla u_i|^2 \,dx,
\end{align}
\begin{align}
\label{boundaryconvex}
  \int_{\partial \Omega} |u| d{\mathscr{H}}^{n-1}\le    \liminf_{i \rightarrow \infty} \int_{\partial \Omega_i} |u_i| d{\mathscr{H}}^{n-1}
\end{align}and
\begin{align}
    \lim_{i \rightarrow \infty} \int_{\Omega_i} fu_i \,dx =\int_{\Omega} fu\,dx.
\end{align}
\end{itemize}
The proof of \eqref{boundaryconvex} relies on the parametrization of
$\partial \Omega$ by the sphere (see also \cite{LP}).

It is challenging to generalize this scheme for convex domains to more rough domains. 
In this context, we formulate the problem for a class of specified rough domains as follows. 

\subsection{Formulation of problem \eqref{min1} over rough domains}
We would like to study the minimization problem \eqref{min1} over some controllable rough domains,
belonging to the class of Sobolev extension domains, with fixed volume. 
A natural class of Sobolev extension domains is the so-called $M$-uniform domain.
In fact, when $n=2$, $M$-uniforms domain are equivalent to extension domains for $H^1$ functions, see 
\cite{Jo} and \cite{GLV}. Recall the following definition of $M$-uniform domain, which was first introduced in \cite{GO} and \cite{Jo}.
\begin{definition}
\label{uniformdomain}
For $M>1$, a domain $\Omega \subset \R^n$ is called {an} $M$-uniform domain if for any $x_1,x_2 \in \overline{\Omega}$, there is a rectifiable curve $\gamma : [0,1] \rightarrow \overline{\Omega}$,
such that $\gamma(0)=x_1, \gamma(1)=x_2$, and
\begin{eqnarray}
\label{Jones}
(i) &\,& {\mathscr{H}}^1(\gamma) \le M|x_1-x_2|, \\
(ii)&\,&  d(\gamma(t), \partial \Omega) \ge \frac{1}{M} \min\{|\gamma(t)-x_1|,|\gamma(t)-x_2|\}, \, \forall t \in [0,1].
\end{eqnarray}
\end{definition}
Roughly speaking, {an} $M$-uniform domain has no interior or exterior cusps, and it does not have very thin connections. The class of $M$-uniform domains contains convex domains in a ball, uniform Lipschitz domains and minimally smooth domain introduced in \cite{St}, and it can have a purely unrectifiable boundary, such as the complement of $4$-corner Cantor set. This class has a wide range of sets.

We remark that if $\Omega\subset B_{R}$ is {an} $M$-uniform domain and $u \in H^1(\Omega)$, then $u$ has 
an extension $\tilde{u}$ which is a BV function in an open neighborhood of $B_{R}$. Thus {if $\Omega$ also has finite perimeter, then} the trace of $u$ can be defined on the reduced boundary $\partial^* \Omega$ in the sense that there exists a measurable function $u^*$ on $\partial^* \Omega$ such that
\begin{align}
\label{cy5'}
\lim_{r \rightarrow 0} \frac{1}{r^n}\int_{B_r(x) \cap \Omega}|u-u^*(x)|\,dy=0, \, {\mathscr{H}}^{n-1}a.e. \,  x \in \partial^* \Omega.
\end{align}We call $u^*$ the (interior) trace of $u$ on $\partial^* \Omega$. The reader can refer to the monograph \cite[Theorem 3.77]{afp}.

Therefore, in the following, we formulate the minimization problem \eqref{min1} over rough sets as the minimization of 
\begin{align}
\label{formulation}
\mathcal{J}_m(u,\Omega):=\frac{1}{2}\int_{\Omega} |\nabla u|^2\,dx+\frac{1}{2m}\left(\int_{\partial^* \Omega} |u^*| \,d{\mathscr{H}}^{n-1}\right)^2-\int_{\Omega} fu \,dx
\end{align}
over all $u \in H^1(\Omega), |\Omega|=V_0$.  We will prove that there is a minimizer to \eqref{formulation} 
{among all sets of $M$-uniform domains with uniformly bounded perimeters}, and thus we are able to solve Problem \eqref{113ex} with\textcolor{red}{in} this class of rough domains.
The $M$-uniform condition of $\Omega$ plays an important role in
generalizing the scheme for convex domains as mentioned above. 
%Without assuming some uniform condition on domains, mathematically it seems undoable to prove the existence of minimizers, because of some modest compactness property should be assumed.  Hence we could not answer Problem \eqref{113ex} fully since we study minimization problem over such restricted class of domains, even though this class contains a wide range of sets. It is quite expected, and it would be very interesting to prove that the appearance of cusps and thin connections would increase the energy $J_m$. 

\subsection{Main Results}
We will first state a theorem asserting the compactness of $M$-uniform domains in $B_{{R}}$, 
{which does not require the domains to have finite perimeters.}

\begin{theorem}
\label{stri}
For $M>0$, let $\{\Omega_i\}$ be a sequence of $M$-uniform domains in $B_{{R}}$ such that 
\begin{align}
\label{assu}
\inf_{i} {\rm{diam}}(\Omega_i)  >0,
\end{align}then there exists {an} $M$-uniform domain $\Omega$ such that
after passing to a subsequence, $\Omega_{i} \rightarrow \Omega$ in $L^1$, as $i\rightarrow \infty$.
\end{theorem}

\begin{remark}
The assumption \eqref{assu} automatically holds if $|\Omega|=V_0>0$, i.e. there is $c=c(V_0,n)>0$ such that ${\rm{diam}}(\Omega) \ge c>0$.
\end{remark}

With the help of Theorem \ref{stri}, we can prove two uniform Poincar\'e inequalities for $M$-uniform domains, see Theorem \ref{unic} and Theorem \ref{dabian} below. 
%Theorem \ref{unic} is more or less the direct application of compactness of $M$-uniform domains, while Theorem \ref{dabian} requires some extra work in a step of proving \eqref{boundaryconvex} which is also in our need to generalize \eqref{unipoincare} for convex domains in the previous scheme. 
Applying Theorem \ref{stri} and Theorem \ref{dabian}, we can prove

\begin{theorem}\label{existence1}
For any $M>0, \Lambda>0, R>0$, and $f\in L^2_{\rm{loc}}(\R^n)$, 
\begin{equation}\label{heat_ins_funs}
\mathcal{J}_m(u,\Omega):=\frac{1}{2}\int_{\Omega} |\nabla u|^2\,dx +\frac{1}{2m}\big(\int_{\partial^* \Omega} |u^*|\,d{\mathscr{H}}^{n-1}\big)^2-\int_{\Omega} fu\,dx.
\end{equation}
Then $\mathcal{J}_m$ admits a minimizer over
\begin{align}
    \mathcal{A}=\Big\{(u,\Omega)\big|u \in H^1(\Omega), \Omega \mbox{\ is {an} $M$-uniform domain in } B_R,
     |\Omega|=V_0>0, {P(\Omega) \le \Lambda} \Big\},
\end{align}
{where $P(\Omega)$ is the perimeter of $\Omega$.}
\end{theorem}

It turns out that \eqref{heat_ins_funs} can also be defined over the space of functions of special bounded variations (or
SBV). 

Let $D\subset\R^n$ be a  bounded smooth domain,
and $f\in L^n(D), f\geq 0$. Consider the following minimization problem:
\begin{equation}
  \inf\Big\{ \mathcal{J}(u):=\frac{1}{2}\int_{\R^n}|\nabla u|^2\, dx+\frac{1}{2m}\big( \int_{J_u} (|u^+|+|u^-|)\, d\mathscr{H}^{n-1} \big)^2-\int_{\R^n}fu\, dx\Big\}
  \label{eqn:SBVmin}
\end{equation}
over $\mathcal{S}=\Big\{u\in {\rm{SBV}}(\R^n, \R_+) \ \big|\ |\left\{ u>0 \right\}|=V_0, \ |\supp u\setminus D|=0, 
\ {\mathscr{H}}^{n-1}(J_u\cap \partial D)=0\Big\}$. Here $\nabla u$ is the absolutely continuous part of
the distributional derivative $Du$ with respect to the Lebesgue measure, and $u^+$ and $u^-$ are one side limit of $u$ on 
the jump set $J_u$ of $u$. See \cite{afp} for the definition of SBV($\R^n$).

In this context, we are able to prove another existence result.

\begin{theorem}\label{existence2}
  $\mathcal{J}(\cdot)$ admits a minimizer $u\in \mathcal{S}$.
  \label{}
\end{theorem}

\begin{remark}
  If $\Omega\subset D$ is \textcolor{red}{an} $M$-uniform domain {of finite perimeter} and $u\in H^1(\Omega)$ is a minimizer of the problem \eqref{min1}, then $u\chi_\Omega\in \mathcal{S}$. On the other hand, for a minimizer $v$ of \eqref{eqn:SBVmin}, if $\Omega:=\left\{ x\in D: v(x)>0 \right\}$ is a subdomain of $D$, and $v$ has no jump in $\Omega$, i.e., $\mathscr{H}^{n-1}(J_{\textcolor{red}{v}} \cap \Omega)=0$, {where $J_v$ is the jump set of $v$,} then $v\in H^1(\Omega)$ and $(v\lfloor_{\Omega}, \Omega)$ is a minimizing pair of problem \eqref{formulation}.  
\end{remark}

\medskip
We will also study Problem \eqref{1133}. This problem is extremely challenging. It seems to be open
that among all $C^2$ domains, if $f \equiv 1$, then whether a ball is an optimal configuration, let alone the uniqueness of an optimal shape. To see some of the difficulties to validate the conjecture, one may compare the
functional $\mathcal{J}_m(u,\Omega)$ with the recently well studied energy functional
\begin{eqnarray}
\widetilde{J}(u,\Omega)=\frac{1}{2}\int_{\Omega} |\nabla u|^2{\, dx} + \beta \int_{\partial \Omega} u^2{\, d\sigma} -\int_{\Omega} u {\, dx},
\end{eqnarray}
where $\beta$ is a positive constant. Due to the linearly splitting property of the regular functional $\widetilde{J}$, Steiner symmetrization argument can be implemented to prove that any smooth optimal domain for $\widetilde{J}$ must be a ball, see the explanation in \cite{BG}. In contrast, it seems that none of the known symmetrization methods can be applied to the minimization problem of $\mathcal{J}_m(u,\Omega)$. 

However, we manage to make some partial progress of Problem \eqref{1133}. 
Our idea is to study this optimization problem through the method of domain variations. 
After some delicate calculations, which involves geometric evolution equations and 
eigenvalue estimate of the Stekloff problem, we prove the following theorem.
\begin{theorem}
\label{new3}
For any $m>0$, $R>0$, and any smooth vector field $\eta \in C_0^{\infty}(\mathbb{R}^n, \mathbb{R}^n)$, 
with $\eta(x)\perp T_x\partial B_R$ for $x\in\partial B_R$, if the flow map $F_t$, associated with $\eta$,
preserves the volume of $B_R$, then $(u_R, B_R)$ is a stable, critical point of $\mathcal{J}_m(\cdot,\cdot)$
in the following sense:
\begin{align} \label{stability2}\frac{d}{dt}\Big|_{t=0}\mathcal{J}_m(u_{F_t(B_R)}, F_t(B_R)) = 0,\quad \frac{d^2}{dt^2}\Big|_{t=0}\mathcal{J}_m(u_{F_t(B_R)}, F_t(B_R)) \ge 0.\end{align}
Here $u_{F_t(B_R)}$ is the unique minimizer of $\mathcal{J}_m(\cdot, F_t(B_R))$ in $H^1(F_t(B_R))$.
\end{theorem}
%We remark that \eqref{eta1} and \eqref{eta2} corresponds to the first and second variation of volume of $B_R$ are zero respectively.

\subsection{Some further remarks}
The compactness of $M$-uniform domains with uniformly bounded perimeter was previously proved by
Li-Wang \cite{LiWang}, where the authors consider the minimization problem arising from the liquid crystal droplet
problem:
\begin{align}
\label{liquiddrop}
    J(u,\Omega):= \int_{\Omega} |\nabla u|^2\, dx + P(\Omega),
\end{align}
where $u \in H^1(\Omega, S^2)$ and $|\Omega|=V_0>0$.
If $(u_i, \Omega_i)$ is a minimizing sequence to \eqref{liquiddrop}, then $\Omega_i$ {automatically} have uniformly bounded perimeters and thus have {an} $L^1$ limit up to a subsequence. It was proven in \cite{LiWang} that the limit is 
$\mathcal{L}^n$-equivalent to {an} $M$-uniform domain. 

Motivated by a volume estimate result in \cite{Packingdimension} for general porous domains, we will show that $M$-uniform domains {turn out to} have uniform{ly bounded} nonlocal perimeters, and thus have an $L^1$ limit 
{up to a subsequence} by the fractional Sobolev compact embedding
theorem, see Corollary \ref{fractionalestimate}. {This together with the argument in \cite{LiWang} yields Theorem \ref{stri}. Hence one may also consider problem \eqref{formulation} over $M$-uniform domains of finite perimeters, without additionally requiring that the perimeters are uniformly bounded as assumed in Theorem \ref{existence1}. The difficulty, however, is that even if there is a limit, and the limit of the domains in the minimizing sequence is still an $M$-uniform domain, it might not have finite perimeter and thus the boundary integral term in \eqref{formulation} may not be well-defined. It would be very interesting to prove that minimizing sequence of \eqref{formulation} do have uniformly bounded perimeters, instead of adding this as an assumption.}

A byproduct of the compactness of $M$-uniform domains is an uniform Poincar\'e inequality for such domains, 
{see Theorem \ref{unic}}. In \cite{BC}, {such} an uniform Poincar\'e inequality {was only proved} for uniformly Lipschitz domains. {Hence} Theorem \ref{unic} generalizes this result of \cite{BC}. 

\subsection{Notations}
Throughout this paper, we adopt the standard notations.
For a set $A \subset \mathbb{R}^n$, we let $A^r:=\{x \in \mathbb{R}^n: d(x, A)<r \}$
and $A_r=\big\{x\in\R^n: B_r(x)\subset A\big\}$. ${\mathscr{H}}^{n-1}$ denotes the $(n-1)$-dimensional Hausdorff measure, {and $d_H(\cdot, \cdot)$ denotes the Hausdorff distance between two sets.} $\mathcal{L}^n$ is the Lebesgue measure in $\mathbb{R}^n$. $|A|$ denotes the Lebesgue measure of $A$. $B_r(x)=\{y \in \mathbb{R}^n: |y-x|<r\}$. 
$\partial^* A$ denotes the reduced boundary of $A$. ${\rm{diam}}(A)$ denotes the diameter of $A$. {Also, we always let $\omega_n$ be the volume of the unit ball in $\mathbb{R}^n$.}

We let $\mathcal{M}_R$ be the class of all $M$-uniform domains contained in $B_R$, and $\mathcal{M}_{R,c}$ be the subclass of {$\mathcal{M}_R$ such that any domain in the subclass has} diameter bigger than or equal to $c>0$. We always use $u^*$ to denote the trace of $u$ in the sense of \eqref{cy5'}. {Last, when we say a set is a domain, we mean the set is a connected open set.}

\section{Preliminaries on rough domains} 

We start with some definitions.
\begin{definition} 
\label{D-c}
For $c>0$, $\mathcal{D}_c$ is the class of sets $E$ satisfying
\begin{eqnarray}
\label{g1}
|B_r(x)\cap E|>cr^n
\end{eqnarray} 
for any $x \in \partial E$ and $0<r<{\rm{diam}}(E)$. 
\end{definition}

The next remark say\textcolor{red}{s} that a\textcolor{red}{ny} set in $\mathcal{D}_c$ is $\mathcal{L}^n$-equivalent to its closure.
\begin{remark}
\label{closure}
If $E \in \mathcal{D}_c$, then $E=\overline{E} \, (mod \, \mathcal{L}^n)$. 
\end{remark}
\begin{proof}
By Lebesgue density theorem, if $E \in \mathcal{D}_c$, then $\partial E \subset E \, (mod\, \mathcal{L}^{n})$.
Hence $|\overline{E}\setminus E|=0$.
\end{proof}
\begin{remark}
\label{feihua1}
If $E \in \mathcal{D}_c$, then for any $x \in \overline{E}$ and $0<r<2{\rm{diam}}(E)$, there is $c'=c'(c,n)>0$ 
such that $|B_r(x) \cap E| \ge c'r^n$.  
\end{remark}

\begin{proof}
There are two cases: \\
(a) If $r \ge 2 d(x, \partial E)$, then there is $z \in \partial E$ and $B_{\frac{r}{2}}(z) \subset B_{r}(x)$, hence, $|B_{r}(x) \cap E| \ge |B_{\frac{r}2}(z) \cap E| \ge c\left(\frac{r}{2}\right)^n=2^{-n}c r^n$.\\
(b) If $r \le 2 d(x,\partial E)$, then $B_{\frac{r}{2}}(x) \subset E$. Thus $|B_{r}(x) \cap E| \ge{\omega_n}\left(\frac{r}{2} \right)^n$. \\
Hence there is $c'=c'(c, n)>0$  such that $|B_r(x) \cap E| \ge c'\epsilon ^n$. 
\end{proof}

The next proposition says $M$-uniform domains belong to the class $\mathcal{D}_c$.
\begin{proposition}
\label{shuyu}
If $\Omega$ is {an} $M$-uniform domain, with ${\rm{diam}}(\Omega) \ge c_0>0$, then $\Omega \in \mathcal{D}_c$ for some $c>0$ depending only on $M$, $n$ and $\frac{{\rm{diam}}(E)}{c_0}$.
\end{proposition}

\begin{proof}
For any $x \in \partial \Omega$ and $0<r<{\rm{diam}}(\Omega)$, we claim that
there is a constant $c_1=c_1(M)>0$  such that there is a ball of radius $c_1 r$ contained in $B_r(x) \cap \Omega$. Indeed, since $0<r<{\rm{diam}}(\Omega)$, there is $y \in \Omega\setminus B_{\frac{r}2}(x)$. Let $\gamma\subset\overline\Omega$ be the curve connecting $x$ and $y$ as in the definition of $M$-uniform domain.
Choose $z \in \partial B_{\frac{1}{3}r}(x) \cap \gamma$. Then $z \in \Omega$ and $d(z, \partial \Omega) \ge \frac{1}{6M}r$. Hence if we choose $c_1(M)=\frac{1}{6M}$, then $B_{c_1(M)r}(z) \subset B_r(x)\cap \Omega$. 
In particular,  for any $x \in \partial \Omega$ and any $0<r<{\rm{diam}} \Omega$, 
$|B_r(x) \cap \Omega| \ge |B_{c_1(M)r}(z)|\ge c_1(M)r^n$.
\end{proof}

The following remark will be used in the proof of compactness of $M$-uniform domains.
\begin{remark}
\label{bukong}
If $\Omega$ is {an} $M$-uniform domain with $|\Omega| \ge c_0$, then there is $r_0>0$ depending only on $M,n,c_0$ such that $\Omega$ contains a ball of radius $r_0$. 
\end{remark}

\begin{proof}
By isodiametric inequality, there is $c_1=c_1(n)>0$ such that ${\rm{diam}}(\Omega) > c_1$. From the proof of Proposition \ref{shuyu}, $\Omega$ contains a ball of radius $\frac{1}{6M} c_1$.  
\end{proof}

Similarly, we define $\mathcal{D}^c$ as follows.
\begin{definition}
\label{D^c}
For $c>0$,  let $\mathcal{D}^c$ be the class of sets $E$ such that
\begin{eqnarray}
\label{g2}
|B_r(x)\cap E^c|>cr^n
\end{eqnarray} holds for any $x \in \partial E$ and $0<r<{\rm{diam}}(E)$. 
\end{definition}

The following proposition is from \cite[Proposition 12.19]{Maggi}. It says that for any set $E \subset \mathbb{R}^n$, we can find an $\mathcal{L}^n$-equivalent set $\widetilde{E}$ with a slightly better topological boundary such that $\partial \widetilde{E}=spt \mu_E$, where $\mu_E$ is the distributional perimeter measure of $E$. 
\begin{proposition}
\label{equiv}
For any Borel set $E \subset \mathbb{R}^n$, there exists an $\mathcal{L}^n$-equivalent set $\widetilde{E}$ such that $|E\Delta \widetilde{E}|=0$ and for any $x \in \partial \widetilde{E}$ and any $r>0$, 
\begin{eqnarray}
\label{spt}
0<|\widetilde{E} \cap B_r(x)|<\textcolor{red}{\omega_n}r^n.
\end{eqnarray}
In other words, $spt \mu_E=spt\mu_{\widetilde{E}}=\partial \widetilde{E}$.
\end{proposition}

The next Lemma concerns the $L^1$-convergence of sets in $\mathcal{D}_c$.
\begin{lemma}
\label{neijin}
Suppose $D_i \subset B_{R_0}$ is a sequence of sets in $\mathcal{D}_c$ such that $D_i \rightarrow D$ in $L^1$. If we identify $D$ with its $\mathcal{L}^n$-equivalent set $\widetilde{D}$ as in Proposition \ref{equiv}, then $D \in \mathcal{D}_c$. Moreover, for any $\epsilon>0$, there is a positive integer $N=N(\epsilon)$ such that for $i>N$, the following properties holds:\\
(i) $D \subset D_i^{\epsilon}$.\\
(ii) $(D_i)_{\epsilon} \subset D$.\\
(iii) $D_i \subset D^{\epsilon}$. \\
In particular, $D_i$ converges to $D$ in the Hausdorff distance, i.e. $d_H(D_i,D) \rightarrow 0$ as $i\to \infty$. 
\end{lemma}

\begin{proof} We argue by contradiction. If (i) were false, then there would exist $x \in D$ such that $B_{\epsilon}(x) \cap D_i=\emptyset$ for $i$ sufficiently large. Hence by the hypothesis and Proposition \ref{equiv}, 
we obtain $0=|B_{\epsilon}(x) \cap D_i|\rightarrow |B_{\epsilon}(x) \cap D|>0$, a contradiction.\\
If (ii) were false, then there would be a sequence $x_i \in (D_i)_{\epsilon} \setminus D$. We may assume $x_i \rightarrow x_0$. Thus $x_0 \in \partial D \cup D^c$. By Proposition \ref{equiv}, we have ${\omega_n}\epsilon^n>|B_{\epsilon}(x_0) \cap D|$. On the other hand, since $B_{\epsilon}(x_i)\subset D_i$, it follows 
\begin{eqnarray*}
|B_{\epsilon}(x_0) \cap D|=\lim_{i \rightarrow \infty}|B_{\epsilon}(x_i) \cap D|
&\ge&\liminf_{i \rightarrow \infty} (|B_{\epsilon}(x_i) \cap D_i|-|D_i \Delta D|)\\
&=&{\omega_n}\epsilon^n-\limsup_{i \rightarrow \infty}|D_i \Delta D|={\omega_n}\epsilon^n,
\end{eqnarray*} 
which is impossible.\\
If (iii) were false, then there would exist a subsequence of $x_i \in D_i\setminus D^{\epsilon}$. Without loss of generality, assume $x_i \rightarrow x_0 \in {\mathbb{R}^n\setminus D^{\epsilon}}$. For any $i$, by Remark \ref{feihua1}, there is $c'>0$ depending only on $c$ and $n$ such that $c'\epsilon ^n \le |B_{\epsilon}(x_i) \cap D_i|.$ On the other hand, since $|B_{\epsilon}(x_0) \cap D|=0$, it follows 
\begin{eqnarray*}
\liminf_{i \rightarrow \infty}|B_{\epsilon}(x_i) \cap D_i| &\le& \limsup_{i\rightarrow \infty}(|B_{\epsilon}(x_i) \cap D|+|D \Delta D_i|) \\
&\le& |B_{\epsilon}(x_0) \cap D|+\limsup_{i \rightarrow \infty} |D_i \Delta D|=0,
\end{eqnarray*}
which is a contradiction.\\
It remains to show $D \in \mathcal{D}_c$. Since $D_i \rightarrow D$ in $L^1$, for any $x \in \partial D$ there is $x_i \in D_i$ such that $x_i \rightarrow x$.  Hence for any $r>0$, {by Remark \ref{feihua1} we have} 
$$|B_r(x) \cap D|=\lim_i |B_r(x_i) \cap D| \ge \liminf_i |B_r(x_i) \cap D_i|-\limsup_i|D_i \Delta D| \ge \textcolor{red}{c'}r^n.$$ Hence $D \in \mathcal{D}_c$.  
\end{proof}

The following remarks {follow} immediately from (i) and (iii) in the above Lemma.

\begin{remark}
\label{new2}
If $D_i$ and $D$ satisfy the same assumption as in Lemma \ref{neijin}, and if ${\rm{int}}(D) \ne \emptyset$, 
then ${\rm{int}}(D)$ is a domain.  If in addition $|{\rm{int}}(D)|=|D|$, then ${\rm{int}}(D) \in \mathcal{D}_c$ 
and $D_i \rightarrow {\rm{int}}(D)$ in $L^1$.
\end{remark}

%\begin{remark}
%\label{nomcompact1}
%Simple examples can show the $int \tilde{D}$ above may not be the $L^1$ limit of $D_i$in $\mathcal{D}_c$. For instance, we choose an arbitrary small positive number $\epsilon$, a dense countable set $\{x_j\}$ in $B_1$ and $\{r_j\} \subset (0,\epsilon)$ such that $\sum_{j=1}^{\infty} r_j^{n-1} \le 1$ and $B_{r_j}(x_j)$ are disjoint.  Then we further choose $0<\rho_j<r_j$ such that for any $x \in B_1$, $x \in \cup_{\alpha=0}^{1/2}F^{\alpha}$m where $F:=\cup_{j=1}^{\infty}B_{\rho_j}(x_j)$. Now let $E_i=B_1 \setminus \cup_{j=1}^i B_{\rho_j}(x_j)$ and $E=B_1 \setminus F$. Then we conclude $E_i \rightarrow E$ in $L^1$, $E_i \in \mathcal{D}_c$, but $int E$ is empty.
%\end{remark}

For sets in $\mathcal{D}^c$, we have the following result, which is similar to Lemma \ref{neijin}. 
\begin{lemma}
\label{waijin}
If $D_i \in \mathcal{D}^c$ and $D_i \rightarrow D$ in $L^1$, and we identify $D$ with its $\mathcal{L}^n$-equivalent set $\widetilde{D}$ as in Proposition \ref{equiv}, then $D \in \mathcal{D}^c$. Moreover, for any $\epsilon>0$, there is a positive integer $N=N(\epsilon)$ such that for $i>N$, the following properties holds:\\
(i) $D \subset D_i^{\epsilon}$.\\
(ii) $(D_i)_{\epsilon} \subset D$.\\
(iii') $D_{\epsilon} \subset D_i$. \\
\end{lemma}

%\begin{example}(Non-compactness of domains in $\mathcal{D}^c$ of uniformly bounded perimeter)\\
%\label{noncompact}
%With the assumptions in Corollary \ref{new1} and we further assume $D_i$ are domains, then the open set $D$ as in Corollary \ref{new1} may not be a domain. For example, consider $D_i \subset \mathbb{R}^2$ to be $B_1((0,0)) \cup B_1((3,0)) \cup \left((1-\frac{1}{i},2+\frac{1}{i}) \times (-\frac{1}{i},\frac{1}{i})\right)$. Then $D_i \in \mathcal{D}^c$, but $D_i$ converge in $L^1$ to $B_1((0,0)) \cup B_1((3,0)) \cup \left([1,2]\times \{0\}\right)$, which is equivalent to $D:=B_1((0,0)) \cup B_1((3,0))$ in the sense of Proposition \ref{equiv}. Hence the limit $D$ is not equivalent to a domain. In fact, one can see any $L^1$ limit or $d^H$ limit of $D_i$ cannot be $\mathcal{L}^n$ equivalent to a domain. This example also says that with the assumptions as in Lemma \ref{waijin}, $d^H(D_i, D)$ does not converge to zero, which contrasts with Lemma \ref{neijin}.
%\end{example}

\section{Proof of Theorem \ref{stri}}

In this section, we will prove Theorem \ref{stri}. We start with the following two Lemmas.

\begin{lemma}
\label{volume1}
Let $\Omega$ be {an} $M$-uniform domain in $B_R\subset \mathbb{R}^n$ with ${\rm{diam}}(\Omega) \ge c_0>0$, then there exists constants $\delta=\delta(M,n) \in (0,1]$ and $C=C(c_0,M,R,n)>0$ such that 
\begin{eqnarray}
\label{starw}
|(\partial \Omega)^r| \le Cr^{\delta}, \quad \forall r \in (0,1].
\end{eqnarray}
\end{lemma}

\begin{lemma}
\label{maincompact}
If $\Omega_i$ is a sequence of $M$-uniform domains in $B_{{R}}$ such that ${\rm{diam}}(\Omega_i) \ge c>0$ and $\Omega_i \rightarrow D$ in $L^1$, then there is {an} $M$-uniform domain $\Omega$ such that $\Omega_i \rightarrow \Omega$ in $L^1$.
\end{lemma}

Lemma \ref{volume1} is essentially proved in \cite{Packingdimension}, where a more general result for porous domains is established. Here we present a simpler proof in the following for reader's convenience. 
The ideas are from \cite{Packingdimension}.

\begin{proof}[Proof of Lemma \ref{volume1}]
Choose $k_0 \ge 1$ such that 
\begin{eqnarray}
\label{k0}
2^{-k_0-1}\le \frac{\min\{c_0, 1\}}{2} \le 2^{-k_0}.
\end{eqnarray}
If $\frac{\min\{c_0, 1\}}{2} \le r \le 1$, then 
\begin{eqnarray}
|(\partial \Omega)^r| \le |B_{R+1}| \le \frac{2|B_{R+1}|}{\min\{c_0, 1\}}r 
\le \frac{2|B_{R+1}|}{\min\{c_0,1\}}r^{\delta}, \, \forall \delta \in (0,1].
\end{eqnarray}
If $0<r\le \frac{\min\{c_0,1\}}{2} $, then we can find some $k \ge k_0$ such that $2^{-k-1} \le r \le 2^{-k}$.

It suffices to prove \eqref{starw} for $r=2^{-k}$, since it would then imply 
\begin{align*}
    |(\partial \Omega)^r| \le C2^{-k \delta}=C(2^{-k-1})^{\frac{k\delta}{k+1}} \le Cr^{\frac{k\delta}{k+1}}\le Cr^{\delta/2}.
\end{align*}
For any $x \in (\partial \Omega)^{2^{-k}}$, there exists $x_1 \in \partial \Omega$ such that $|x-x_1| <2^{-k}$. Then for any $k_0 \le j \le k$, by the choice of $k_0$ in \eqref{k0}, ${\rm{diam}}(\Omega) > 2^{-j+1}$ so that
there exists $x_2 \in \partial B_{2^{-j+1}}(x_1) \cap \overline{\Omega}$. Let $\gamma\subset\overline\Omega$ be the path connecting $x_1$ and $x_2$ as in Definition \ref{uniformdomain}. Let $y \in \partial B_{2^{-j}}(x_1) \cap \gamma$, and thus 
\begin{align}
    \label{j1w}
    d(y,\partial \Omega) \ge \frac{1}{M}\min\{|y-x_1|, |y-x_2|\}=\frac{2^{-j}}{M}.
\end{align}
We cover $B_R \setminus \partial \Omega$ by $\left\{B_{r_z}(z): z \in B_R \setminus \partial \Omega, r_z=\frac{d(z,\partial\Omega)}{15}\right\}:=\mathcal{B}_1$. By Vitalli's covering Lemma, we can choose a countable pairwise disjoint subfamily $\mathcal{B}$ of $\mathcal{B}_1$ such that 
\begin{align*}
    B_R \setminus \partial \Omega \subset \cup_{B \in \mathcal{B}} 5B.
\end{align*}Hence $y \in B_{5r_z} (z)$ for some $B_{r_z}(z) \in \mathcal{B}$. 

Clearly, 
\begin{align*}
    d(z, \partial \Omega) \le |z-x_1| \le |z-y|+|y-x_1| \le 5r_z+2^{-j} =\frac{1}{3}d(z,\partial \Omega)+2^{-j},
\end{align*}
which implies
\begin{align*}
    d(z,\partial \Omega) \le \frac{3}{2}2^{-j}, \ \
    5r_z \le 2^{-j-1}.
\end{align*}
Therefore, 
\begin{align}
    \label{j3}
    z \in B_{2^{-j+1}}(x_1) \setminus B_{2^{-j-1}}(x_1). 
\end{align}

Notice that by \eqref{j1w}, it follows from $y\in B_{5r_z}(z)$ that
\begin{align*}
    \frac{2^{-j}}{M} \le 20r_z,
\end{align*}
and hence 
\begin{eqnarray*}
    |x-z| \le |x-x_1|+|x_1-y|+|y-z| &< &2^{-k}+2^{-j}+5r_z \le 2^{-j+1}+5r_z\\
    & \le & (40M+5)r_z \le 45Mr_z.
\end{eqnarray*}
Therefore, $x \in B_{45Mr_z}(z)$.

So far, we have shown that for any $x \in (\partial \Omega)^{2^{-k}}$ and $k_0 \le j \le k$, there is 
$z_j \in B_{2^{-j+1}}(x_1) \setminus B_{2^{-j-1}}(x_1)$ such that $x \in B_{45Mr_{z_j}}(z_j)$ and $B_{r_{z_j}}(z_j) \in \mathcal{B}$. Therefore, $\forall x \in (\partial \Omega)^{2^{-k}}$, we have 
\begin{align}
    \label{keytool}
    \sum_{B \in \mathcal{B}}\chi_{45MB}(x) \ge \frac{k-k_0}{3},
\end{align}
since by \eqref{j3} each $B \in \mathcal{B}$ can be considered at most 3 times in order that $x \in 45MB$.

By Hardy-Littlewood Theorem,  there is constant $c_n \ge 1$ such that for any $p>1$, 
\begin{align}
    \label{hl}
    \Vert \mathcal{M}\phi \Vert_{L^p} \le c_n \left(\frac{p}{p-1}\right)^{1/p}\Vert \phi \Vert_{L^p},
\end{align}where $\mathcal{M}\phi$ is the non-centered Hardy-Littlewood maximal function.

Let $\delta=\frac{1}{9(45M)^nc_n}$. By \eqref{keytool} we have
\begin{align*}
    |(\partial \Omega)^{2^{-k}}| =& 2^{-k\delta}\int_{(\partial \Omega)^{2^{-k}}} 2^{k\delta} \\
    \le & 2^{-k\delta} \int_{(\partial \Omega)^{2^{-k}}}2^{(k_0+3 \sum_{B \in \mathcal{B}}\chi_{45MB}(x))\delta}dx\\
    \le & 2^{-k\delta}2^{k_0 \delta}\int_{B_R} \sum_{m=0}^{\infty} \frac{(3 \delta \sum_{B \in \mathcal{B}}\chi_{45MB}(x))^m}{m!}dx
\end{align*}

For any nonnegative $\phi \in L^{\frac{m}{m-1}}$, $m>1$, we have 
\begin{align*}
    \int \phi(x) \sum_{B \in \mathcal{B}}\chi_{45MB}(x) dx \le & (45M)^n \sum_{B \in \mathcal{B}} |B| \frac{1}{|45MB|}\int_{45MB} \phi(x) dx\\
    \le & (45M)^n \sum_{B \in \mathcal{B}}\int_B \mathcal{M}\phi(x) dx\\
    \le & (45M)^n \Vert M\phi \Vert_{L^{\frac{m}{m-1}}}\left(\int (\sum_{B\in \mathcal{B}}\chi_B(x))^m dx\right)^{1/m}\\
    \le & (45M)^nc_n m |B_{2R}|^{\frac{1}m} \Vert \phi \Vert_{L^{\frac{m}{m-1}}}.
\end{align*}
Hence by duality, for $m>1$ we obtain
\begin{align}
    \label{j4}
    \Big\|\sum_{B \in \mathcal{B}}\chi_{45MB}\Big\|_{L^m} \le (45M)^n c_n m |B_{2R}|^{\frac{1}m}.
\end{align}
It is straightforward to verify \eqref{j4} for $m=1$. Therefore,
\begin{align*}
    |(\partial \Omega)^{2^{-k}}| \le & 2^{-k\delta}2^{k_0\delta}|B_{2R}|\sum_{l=0}^{\infty}\frac{\left(3(45M)^n\delta c_n l\right)^l}{l!}\\
    \le & 2^{-k\delta}2^{k_0\delta}|B_{2R}|\sum_{l=0}^{\infty} \left(\frac{e}{3}\right)^l ,\quad \mbox{by Stirling's formula and the choice of $\delta$}\\
    = & C(k_0, R, \delta, n) 2^{-k\delta}\\
    \le & C(c_0,M,R,n) 2^{-k\delta}, \, \mbox{since $k_0$ depends on $c_0$.}
\end{align*}
This completes the proof.
\end{proof}

%Let $\mathcal{M}_R$ denotes the class of $M$-uniform domains in $B_R \subset \mathbb{R}^n$. We have the following corollary from Theorem \ref{volume}, thanks to \cite{Beer}. This provides a new $L^1$ compactness result for certain class of domains without using uniformly bounded perimeter condition, since $M$-uniform domains can have unbounded perimeter.

Lemma \ref{volume1} yields the following Corollary.
\begin{corollary}
\label{fractionalestimate}
Let $\Omega$ be {an} $M$-uniform domain in $B_R\subset \mathbb{R}^n$ with ${\rm{diam}}(\Omega) \ge c_0>0$. Then there exists a constant $\delta=\delta(M,n) \in (0,1]$ such that for any $s \in (0,\delta)$, 
\begin{align}
\label{sperimeterestimate}
    \big[\chi_{\Omega}\big]_{W^{s,1}(B_R)} \le C=C(M,n,R,s,c_0).
\end{align}
\end{corollary}
\begin{proof}
Let $\delta$ be as in Lemma \ref{volume1}. Then \eqref{sperimeterestimate} follows from the following estimate
\begin{align*}
    \int_{B_R}\int_{B_R} \frac{|\chi_{\Omega}(x)-\chi_{\Omega}(y)|}{|x-y|^{n+s}} dy dx=&\int_{B_R}\int_0^{2R}\int_{\partial B_r(x)} \frac{|\chi_{\Omega}(x)-\chi_{\Omega}(y)|}{r^{n+s}} d{\mathscr{H}}^{n-1}(y)dr dx \\
    = &\int_0^{2R}\int_{(\partial \Omega)^r}\int_{\partial B_r(x)} \frac{|\chi_{\Omega}(x)-\chi_{\Omega}(y)|}{r^{n+s}} d{\mathscr{H}}^{n-1}(y)dxdr\\
    \le  &\int_0^{2R}\int_{(\partial \Omega)^r}\int_{\partial B_r(x)} \frac{1}{r^{n+s}} d{\mathscr{H}}^{n-1}(y) dxdr\\
    \le & \int_0^{2R} Cr^{\delta} r^{-s-1} dr%, \quad \mbox{by Lemma \ref{volume1}}\\
    \le C(M,n,R,s,c_0)<\infty,
\end{align*}where in the second equality we have used that if $x \notin (\partial \Omega)^r$ and $y \in B_r(x)$, 
then $\chi_{\Omega}(x)=\chi_{\Omega}(y).$
\end{proof}

Next, we prove Lemma \ref{maincompact}.
\begin{proof}[Proof of Lemma \ref{maincompact}]
Without loss of generality, we may assume $spt \mu_D=\partial D$ as in Proposition \ref{equiv}. 
We first prove that ${\rm{int}}(D) \ne \emptyset$. Indeed, notice that by Remark \ref{bukong}, each $\Omega_i$ contains a fixed ball of radius $r_0$ depending only on $c_0,n$ and $M$. Therefore, for each $\Omega_i$, if $\epsilon<\frac{r_0}2$, then by definition $(\Omega_i)_{\epsilon}$ contains a ball of radius $\frac{r_0}2$. By Lemma \ref{neijin} (ii), $D$ also contains a ball of radius $\frac{r_0}2$. In particular, ${\rm{int}}(D) \ne \emptyset$.

Now let $\Omega={\rm{int}}(D)$. It suffices to show $\Omega$ is {an} $M$-uniform domain, since the $L^1$ convergence in the statement can then be directly deduced from Remark \ref{closure}, Proposition \ref{shuyu} and the fact $\Omega \subset D \subset \overline{\Omega}$.

Fix any $x,y \in \Omega$, then for any given $N>2M$, we may choose $0<\epsilon<\frac{1}{N}$ so small that $k\epsilon < d(x, \partial \Omega) \le (k+1)\epsilon$ for some $k>(1+\frac{1}{M})(N+1)$, and $|x-y|>2(N+1)\epsilon$. Since ${\rm{int}}(\Omega) \ne \emptyset$,
it follows from Lemma \ref{neijin} (i) and (iii) that $d_H(\Omega_i, \Omega) \rightarrow 0$.
Hence we can find $x_i, y_i \in \textcolor{red}{\Omega_i}$, with $|x_i-x|<\epsilon, |y_i-y|<\epsilon$ for $i$ large. By Lemma \ref{neijin} (ii), we may choose $i$ so large that 
\begin{eqnarray}
\label{cru}
(\Omega_i)_{\epsilon} \subset \Omega.
\end{eqnarray} Also we choose $\gamma_i \subset \Omega_i$ to be the rectifiable curve connecting $x_i$ and $y_i$ in $\Omega_i$ as in the definition of $M$-uniform domain. For any $p \in \gamma_i$, if $p \in B_{N \epsilon}(x_i) \cup B_{N\epsilon}(y_i)$, then clearly $p \in B_{(N+1) \epsilon}(x) \cup B_{(N+1)\epsilon}(y) \subset \Omega$. Moreover, this implies
\begin{equation}
\label{a1}
d(p, \partial \Omega) \ge k\epsilon-(N+1) \epsilon>\frac{1}{M}(N+1)\epsilon\ge \frac{1}{M}\min\big\{|p-x|,\ |p-y|\big\}
\end{equation}
Clearly, \eqref{a1} also holds for any $p$ on the line segment between $x_i$ and $x$, and between $y_i$ and $y$.
If $p \notin B_{N \epsilon}(x_i) \cup B_{N\epsilon}(y_i)$, then 
$$d(p, \partial \Omega_i) \ge \frac{1}{M} \min\big\{|p-x_i|,|p-y_i|\big\}> \frac{N\epsilon}{M},$$
thus $p \in (\Omega_i)_{\frac{N\epsilon}M} \subset (\Omega_i)_{\epsilon}\subset \Omega \cap \Omega_i$. Moreover, let $r=d(p,\partial((\Omega_i)_{\epsilon}))$, then by \eqref{cru}, $B_r(p) \subset \Omega$, so $d(p,\partial \Omega) \ge r=d\left(p, \partial ((\Omega_i)_{\epsilon})\right) \ge d(p,\partial \Omega_i)-\epsilon$. 
Therefore,
\begin{equation}
\label{a3}
\frac{d(p,\partial \Omega)}{\min\{|p-x_i|,|p-y_i|\}} \ge \frac{d(p,\partial \Omega_i)-\epsilon}{\min\{|p-x_i|,|p-y_i|\}}\ge \frac{1}{M}-\frac{\epsilon}{N\epsilon}\ge \frac{1}{M}-\frac{1}{N}.\\
\end{equation}
Hence by the choice of $\epsilon$ and $N$, it follows
\begin{equation}
\label{a2}
d(p,\partial \Omega) \ge (\frac{1}{M}-\frac{1}{N})(\min\{|p-x|,|p-y|\}-\epsilon) \ge(\frac{1}{M}-\frac{1}{N})(\min\{|p-x|,|p-y|\})-\frac{1}{MN}.
\end{equation}

Therefore, we may let $\gamma^N$ be the curve that consists of the following three parts. The first part {is a line segment starting from $x$ to $x_i$}, the second part {is the curve $\gamma_i$ found above, which starts from $x_i$ to $y_i$,} and the third part {is a line segment starting from $y_i$ to $y$.}

It is clear from the discussion above that 
$\gamma^N \subset \Omega$ and $\gamma^N$ {starts from $x$ to $y$.}
Moreover, from \eqref{a1} and \eqref{a2} and the choice of $\epsilon$, we obtain that \\
\begin{eqnarray*}
{\mathscr{H}}^1(\gamma^N) &\le& M|x_i-y_i|+|x_i-x|+|y_i-y|\\
&\le& M|x-y|+(M+1)|x_i-x|+(M+1)|y_i-y|\\
&\le& M|x-y|+2\frac{M+1}{N},
\end{eqnarray*}
and
$$
d(p, \partial \Omega) \ge (\frac{1}{M}-\frac{1}{N}) \min\big\{|p-x|,\ |p-y|\big\}-\frac{1}{MN},
\quad \forall p\in \gamma^N.
$$
Then by the compactness of $(\overline{\Omega}, d_H)$, and since $\gamma^N$ is connected, 
there is a compact connected set $E \subset \overline{\Omega}$ such that 
$d_H(\gamma^N, E) \rightarrow 0$ as $N \rightarrow \infty$. 
Then by \cite[Theorem 3.18]{Fa}, 
$${\mathscr{H}}^1(E) \le \liminf_{N \rightarrow \infty} {\mathscr{H}}^1(\gamma^N) \le M|x-y|.$$
Hence by \cite[Lemma 3.12]{Fa}, $E$ is arc-wise connected so that we can choose a rectifiable
curve $\gamma \subset E$ joining $x$ and $y$. 
For any $p \in \gamma$, we can choose sequence $p_N\in \gamma^N, p_N \rightarrow p$. Since 
$$d(p_N, \partial \Omega) \ge (\frac{1}{M}-\frac{1}{N}) \min\{|p_N-x|,|p_N-y|\}-{\frac{1}{MN}},$$
it follows by passing to the limit $N \rightarrow \infty$ that
$$d(p, \partial \Omega) \ge \frac{1}{M} \min\{|p-x|,|p-y|\},$$ 
which also implies $\gamma \subset {\rm{int}}(\Omega)$. 
Therefore $\gamma$ satisfies both properties in the definition of $M$-uniform domain, 
and $\Omega$ is $M$-uniform. By Corollary \ref{new2} and Proposition \ref{shuyu}, $\Omega$ is a domain. 
This finishes the proof.
\end{proof}

Now we are ready to prove Theorem \ref{stri}.
\begin{proof}[Proof of Theorem \ref{stri}]
By Corollary \ref{fractionalestimate}, the sequence $\chi_{\Omega_i}$ are uniformly bounded in $W^{s,1}\textcolor{red}{(B_R)}$.
By the compact embedding from $W^{s,1}(B_R)$ to $L^q(B_R)$ space with $1 \le q \le 1^*:=\frac{n}{n-s}$, we
conclude that there exists a subsequence of $\Omega_i$ that converges to a set $D \subset B_R$ in $L^1$. By Lemma \ref{maincompact}, $D$ is $L^1$ equivalent to an $M$-uniform domain. This finishes the proof.
\end{proof}

\section{Uniform Poincar\'e inequality and existence of minimizer to \eqref{formulation}}

In this section, we will apply Theorem \ref{stri} to deduce two uniform Poincar\'e inequalities {via compactness argument, and then we will prove Theorem \ref{existence1}.}

\begin{theorem}
\label{unic} For any domain $\Omega\in\mathcal{M}_R$, 
there exists a constant $C>0$ depending on $M, R$ such that 
\begin{eqnarray}
\label{poincare1}
\int_{\Omega} u^2 \,dx \le C\int_{\Omega} |\nabla u|^2 \,dx, \ \ \, \forall u \in H^1(\Omega) \ {\rm{with}}\ \int_{\Omega} u \,dx =0.
\end{eqnarray}
\end{theorem}

\begin{proof}
We divide the proof of \eqref{poincare1} for $\Omega \in \mathcal{M}_R$ into two cases:\\
(i) If ${\rm{diam}}(\Omega) \ge 1$, {then} we argue by contradiction. Suppose there exist pairs $(\Omega_i, u_i)$ such that $\Omega_i \in \mathcal{M}_R$, ${\rm{diam}}(\Omega_i) \ge 1$, $u_i \in H^1(\Omega_i)$ satisfies
$$\int_{\Omega_i} u_i\,dx=0,\ \ \int_{\Omega_i} u_i^2\,dx =1,$$ 
but
$$\int_{\Omega_i} |\nabla u_i|^2 dx \rightarrow 0\ {\rm{as}}\ i \rightarrow \infty.$$
Let $\widetilde{u}_i$ be an extension of $u_i$ such that 
$$\|\widetilde{u}_i\|_{H^1(B_R)} \le C(M,n) \Vert u_i\Vert_{H^1(\Omega_i)}.$$
Hence $\{\tilde{u}_i\}$ is a bounded sequence in $H^1(B_R)$. {Hence} we may assume that there exists $u\in H^1(B_R)$ such that 
$\widetilde{u}_i\rightharpoonup u$ in $H^1(B_R)$ and $\widetilde{u}_i \rightarrow u$ in $L^2(B_R)$. 
By Theorem \ref{stri}, there is {an} $M$-uniform domain $\Omega \in \mathcal{M}_R$ such that 
$\Omega_i \rightarrow \Omega$ in $L^1$. 

{Since $\chi_{\Omega_i}\nabla \tilde{u}_i \rightharpoonup \chi_{\Omega}\nabla \tilde{u}$ weakly in $L^2$}, by the lower semicontinuity {property of weak convergence}, we have 
\begin{align*}
    \int_{\Omega} |\nabla u|^2 \,dx \le \liminf_{i \rightarrow \infty} \int_{\Omega_i}|\nabla u_i|^2 \,dx=0.
\end{align*}
Hence $u \equiv c$ in $\Omega$. On the other hand,
\begin{align*}
    |\int_{\Omega_i} u_i^2 \,dx -\int_{\Omega} u^2 \,dx| \le & |\int_{\Omega_i} u_i^2 \,dx-\int_{\Omega_i} u^2 \,dx|+|\int_{\Omega_i} u^2 \,dx-\int_{\Omega} u^2 \,dx|\\
    \le & \Vert\widetilde{u}_i+u\Vert_{L^2(B_R)}\Vert \widetilde{u}_i-u\Vert_{L^2(B_R)}+\int_{\Omega_i \Delta \Omega} u^2 dx\\
    \rightarrow & 0,  \ \ \ \mbox{as $i \rightarrow \infty$.}
\end{align*}
Hence 
\begin{align}
    \label{kkk0}
    \int_{\Omega} u^2 dx=1.
\end{align}
Similarly, we have $\int_{\Omega} u \,dx= \lim_{i \rightarrow \infty} \int_{\Omega_i} u_i \,dx=0$. 
Hence $c=0$ and $\int_{\Omega} u^2 \,dx=0$. This contradicts \eqref{kkk0}. 
Therefore, we have proved \eqref{poincare1}. \\
(ii) {If} ${\rm{diam}}(\Omega) < 1$, {then} we may assume {that} $0 \in \Omega$. Hence we can choose a $0<t<1$ such that 
$\Omega_t:=\frac{1}{t}\Omega \in \mathcal{M}_R$ with ${\rm{diam}}(\Omega_t) =1$. 
For any $u \in H^1(\Omega)$ with $\int_\Omega u\,dx=0$, from (i) we then have
\begin{align*}
    \int_{\Omega} u^2(x) \,dx = &t^n\int_{\Omega_t} u^2(tx) \,dx
    \le C t^n \int_{\Omega_t} |\nabla(u(tx))|^2 \,dx \\
    =&C t^{n+2} \int_{\Omega_t} {|\nabla u(tx)|^2} \,dx=Ct^2\int_{\Omega}|\nabla u|^2\,dx \\
    \le & C\int_{\Omega} |\nabla u(x)|^2 \,dx,
\end{align*}
since $0<t<1$. This finishes the proof.
\end{proof}

The second uniform Poincar\'e inequality has a slightly different form, 
which will be useful to prove the existence of the minimization problem \eqref{formulation}.
\begin{theorem}
\label{dabian} For any $\Omega\in\mathcal{M}_{R,c}$ {with $P(\Omega) \le \Lambda$}, there exists a constant $C>0$ 
depending on $M$, $c$, {$\Lambda$} and $R$ such that 
\begin{eqnarray}
\label{poincares}
\int_{\Omega} u^2 \,dx \le C \Big(\int_{\Omega} |\nabla u|^2 \,dx+\big(\int_{\partial^* \Omega} |u^*(x)| \,d{\mathscr{H}}^{n-1} \big)^2 \Big), \ \ \ \forall u \in H^1(\Omega).
\end{eqnarray}
\end{theorem}

%\textcolor{blue}{Here, I have removed Theorem 4.3 in the original version, since truncation argument is no longer need if assuming uniform bound of perimeters.}
\begin{proof}
Suppose \eqref{poincares} were false. Then by scaling, we may assume that there would exist pairs $(\Omega_i, u_i)$ such that $\Omega_i \in \mathcal{M}_{R,c}$, {$P(\Omega_i) \le \Lambda$,} ${\rm{diam}}(\Omega_i) \ge c$, $u_i \in H^1(\Omega_i)$ such that
$$\int_{\Omega_i} u_i^2 =1,$$
but 
$$\int_{\Omega_i} |\nabla u_i|^2\,dx + \big(\int_{\partial^* \Omega} |u^*| \,d{\mathscr{H}}^{n-1}\big)^2 \rightarrow 0
\ \ {\rm{as}}\  \ i \rightarrow \infty.$$
We may assume for convenience that $u_i \ge 0$. Let $\widetilde{u}_i$ be an extension of $u_i$ such that 
$$\|\widetilde{u}_i\|_{H^1(B_{{R}})} \le C(M,n) \|u_i\|_{H^1(\Omega_i)}.$$
Hence $\{\widetilde{u}_i\}$ is a bounded sequence in $H^1(B_R)$.
Let $u\in H^1(B_R)$ be the weak limit of $\widetilde{u}_i$ in $H^1(B_R)$
and $\widetilde{u}_i \rightarrow {u}$ in $L^2(B_{R})$. 
By Theorem \ref{stri} {and lower semicontinuity of sets of finite perimeter}, there is {an} $M$-uniform domain $\Omega \in \mathcal{M}_{R,c}$ {with $P(\Omega)\le \Lambda$} such that $\Omega_i \rightarrow \Omega$ in $L^1$. 

{As in the proof of Theorem \ref{unic}, we have that
\begin{align*}
    \int_{\Omega} |\nabla u|^2 dx \le \liminf_{i \rightarrow \infty} \int_{\Omega_i} |\nabla u_i|^2 dx=0,
\end{align*}and thus $u \equiv c$ in $\Omega$ for some constant $c$. Also,
\begin{align}
    \label{kkk}
    \int_{\Omega} u^2 dx=1.
\end{align}}
Now let $\Bar{u}_i=\widetilde{u}_i \chi_{\Omega_i}$ and $\bar{u}={u}\chi_{\Omega}$. By \cite[Theorem 3.84]{afp} and the structure of $BV$ function, we know that $\Bar{u}_i, u \in SBV(\mathbb{R}^n)$, with
\begin{align*}
    J_{\Bar{u}_i}=  \partial ^* \Omega_i \cap \{u_i^*>0\}
\end{align*}
and 
\begin{align*}
    J_u=\partial ^* \Omega \cap \{u^* >0\}.
\end{align*}
Here $J_u$ denotes the measure theoretical jump part of a BV function $u$. 

We let $w^-$ and $w^+$ denote the measure theoretical interior and exterior trace of a BV function $w$ on $\partial^* \Omega$ respectively. Since $\mathscr{H}^{n-1}(\partial^* \Omega_i)\le \Lambda$, 
{and that $\nabla \widetilde{u}_i\chi_{\Omega_i} \rightharpoonup \nabla \widetilde{u}\chi_{\Omega}$ weakly in $L^2(B_R)$}, we can apply \cite[Theorem 2.3 and Theorem 2.12]{B} to obtain
\begin{align*}
\int_{\partial^* \Omega} u^* \,d{\mathscr{H}}^{n-1}= &
    \int_{J_u} |u^--u^+| \,d{\mathscr{H}}^{n-1}\\
    \le& \liminf_{i \rightarrow \infty} \int_{J_{\Bar{u}_i}} |\Bar{u}_i^--\Bar{u}_i^+| \,d{\mathscr{H}}^{n-1}\\
    =&  \liminf_{i \rightarrow \infty}\int_{\partial^* \Omega_i} u_i^*.
\end{align*}
Hence $\int_{\partial^* \Omega} u^*\,d{\mathscr{H}}^{n-1}=0$
and  $u \equiv 0$ in $\Omega$. This contradicts \eqref{kkk}.
\end{proof}

Now we are ready to give a proof of Theorem \ref{existence1}:
\begin{proof}[Proof of Theorem \ref{existence1}]
Let $(u_i,\Omega_i)$ be a minimizing sequence, and we may assume that 
$u_i$ is a minimizer of $\mathcal{J}_m(\cdot, \Omega_i)$ among all $H^1(\Omega_i)$ functions.
From ${\mathcal{J}_m}(u_i,\Omega_i) \le {\mathcal{J}_m}(0,\Omega_i)=0$, we deduce that
\begin{align*}
& \int_{\Omega_i} |\nabla u_i|^2\,dx+\frac{1}{2m} \big(\int_{\partial \Omega_i}u_i\,d{\mathscr{H}}^{n-1}\big)^2\le  \int_{\Omega_i}f u_i \,dx
\le \epsilon \int_{\Omega_i} u_i^2\,dx+ C_{\epsilon} \int_{\Omega} f^2\,dx\\
& \le C\epsilon  \Big(\int_{\Omega} |\nabla u_i|^2 \,dx+\big(\int_{\partial^* \Omega} |u_i^*| \,d{\mathscr{H}}^{n-1}\big)^2 \Big)+C_{\epsilon} \int_{\Omega} f^2\,dx, 
\end{align*}
where we have used Theorem {\ref{dabian}}. By choosing a small $\epsilon>0$, this implies that 
\begin{align}
\label{xuexi}
    \sup_i \Big(\int_{\Omega_i} |\nabla u_i|^2\,dx+\int_{\partial \Omega_i}u_i\,d{\mathscr{H}}^{n-1}\Big) <\infty.
\end{align}
Hence the infimum of ${\mathcal{J}_m}>-\infty$. Moreover, by Theorem \ref{dabian} and \eqref{xuexi},
$$\sup_i||u_i||_{H^1(\Omega_i)} < \infty.$$
Now we can {repeat} the same argument as in the proof of Theorem \ref{dabian} to 
conclude that there exists a $(u,\Omega) \in \mathcal{A}$ such that 
$$\mathcal{{J}}_m(u,\Omega) \le \liminf_{i \rightarrow \infty} \mathcal{{J}}_m(u_i,\Omega_i).$$ 
The proof is completed.\end{proof}

%	Define $\zeta$ as the harmonic extension of $\eta\cdot \nu_{t}$, i.e., 
%	\begin{equation}
%	\left\{
%	\begin{array}{ll}
%	\Delta \zeta=0 & \text{ in $\Omega_t$}, \\
%	\zeta=\eta\cdot \nu & \text{ on $\partial \Omega_t$}.
%	\end{array}
%	\right.
%	\label{}
%	\end{equation}

\section{Existence of minimizers in SBV}
In this section we will extend the existence results in the previous section to the setting of SBV, and prove
Theorem \ref{existence2}. The argument of our proof is similar to that by \cite{BG10}.

\begin{proof} [Proof of Theorem \ref{existence2}]
 We prove it by the direct method of calculus of variation. \\
\noindent{\it Claim} 1: $\mathcal{J}$ is bounded from below on $\mathcal{S}$. \\
 For any $u\in \mathcal{S}$, since $\supp u\subset D$ and $\mathscr{H}^{n-1}(J_u\cap\partial D)=0$, we have the following Sobolev type inequality (\cite[Theorem 4.10]{MZ}):
 \begin{equation}
   \left\|u\right\|_{L^{\frac{n}{n-1}}(D)}\leq C\left|Du\right|(D). 
   \label{eqn:BVSobolev}
 \end{equation}
 From \eqref{eqn:BVSobolev},  Young's inequality and the fact that $t^2>t-1$, we can derive
 \begin{equation}
 \begin{split}
   \mathcal{J}(u)&\geq \frac{1}{4}\int_D|\nabla u|^2\,dx+\frac{1}{4m}\Big( \int_{J_u}(|u^+|+|u^-|)\, d\mathscr{H}^{n-1} \Big)^2\\&+\frac{1}{4}\int_{D}(|\nabla u|-1)\, dx+\frac{1}{4m}\Big(\int_{J_u}(|u^+|+|u^-|)\, d \mathscr{H}^{n-1}-1\Big)-\int_{D} fu\, dx\\   
   &\geq\frac{1}{4}\int_D|\nabla u|^2\,dx+\frac{1}{4m}\Big( \int_{J_u}(|u^+|+|u^-|)\, d\mathscr{H}^{n-1} \Big)^2\\
   &+ C\Big( \int_{D}|\nabla u|\,dx+\int_{J_u}(|u^+-u^-|)\, d\H^{n-1} \Big)-C-\int_{D}fu\, dx\\
   &=\frac{1}{4}\int_D|\nabla u|^2\,dx+\frac{1}{4m}\Big( \int_{J_u}(|u^+|+|u^-|)\, d\mathscr{H}^{n-1} \Big)^2\\
   &+C|Du|(D)-C-\int_{D}f u\, dx\\
   &\geq\frac{1}{4}\int_D|\nabla u|^2\,dx+\frac{1}{4m}\Big( \int_{J_u}(|u^+|+|u^-|)\, d\mathscr{H}^{n-1} \Big)^2
   + C\left|Du\right|(D)\\
   &-C-\epsilon \big\|u\big\|_{L^{\frac{n}{n-1}}(D)}-C(\epsilon)\big\|f\big\|_{L^n(D)}\\
    &\geq -C-C\left\|f\right\|_{L^n(D)},
 \end{split}
 \label{eqn:SBVlowerbound}
 \end{equation}
 provided $\epsilon$ is chosen sufficiently small. 
 Hence the functional $\mathcal{J}$ is bounded from below,
 and we can find a minimizing sequence $\{u_i\}$ in $\mathcal{S}$ such that 
 \begin{equation}
 	\lim_{i\rightarrow\infty} \mathcal{J}(u_i)=\inf_{u\in\mathcal{S}} \mathcal{J}(u)>-\infty.
 \end{equation}

\noindent{\it Claim 2}. There exists $u\in {\rm{SBV}}(D)$ such that after taking a subsequence,
$u_i\rightharpoonup u$ in $BV$.
From the penultimate inequality of \eqref{eqn:SBVlowerbound} we have

\begin{equation}
  \begin{split}
    &\sup_i\left\|u_i\right\|_{BV(D)}=\sup_i\big( |Du_i|(D)+\big\|u_i\big\|_{L^1(D)} \big)\\
    &\leq C\sup_i \big(\mathcal{J}(u_i)+C+C\left\|f\right\|_{L^n(D)}\big)<\infty,
  \end{split}
  \label{eqn:uniformBVbound}
\end{equation}
and
\begin{equation}
\begin{split}
  &\sup_i\Big( \int_{D}|\nabla u_i|^2\,dx+\int_{J_{u_i}}(|u_i^+|+|u_i^-|)\, d\mathscr{H}^{n-1} \Big)\\
  &\leq C\sup_{i}\left( \mathcal{J}(u_i)+C+C\left\|f\right\|_{L^n(D)} \right)<\infty.
  \end{split}
  \label{eqn:uniformL2bound}
\end{equation}

By the compactness theorem of $BV$ functions (\cite[Theorem 3.23]{afp}),
there exists a subsequence $\left\{ u_{i_k} \right\}$ and $u\in BV(D)$ such that
$u_{i_k}\rightharpoonup u$ in $BV(D)$, i.e., 
 \begin{equation}
 \begin{cases}
 	u_{i_k}\rightarrow u\ \ \text{ in }\ \  L^1(D),\\
 	\label{eqn:L1conv}
Du_{i_k}\overset{\ast}{\rightharpoonup} Du\ \ \text{ in }\ \ \mathcal{M}(D).
\end{cases}
\end{equation}
 For every $\epsilon>0$, let $u_{i_k}^\epsilon:=\max\{u_{i_k}, \epsilon\}, u^\epsilon:=\max\left\{ u, \epsilon \right\}$. Then we have
 \begin{equation}
 	u_{i_k}^\epsilon{\rightharpoonup} u^\epsilon \text{ in }BV(D).
 	\label{eqn:SBV1}
 \end{equation}
  From \eqref{eqn:uniformL2bound} we have
 \begin{equation}
   \sup_{k}\int_{D}|\nabla u_{i_k}^\epsilon|^2=\sup_{k}\int_{D}\left|\nabla u_{i_k}\chi_{ \left\{ u_{i_k}>\epsilon \right\}}\right|^2\leq \sup_k\int_{D}|\nabla u_{i_k}|^2<\infty. 
   \label{eqn:SBV2}
 \end{equation}
Moreover, from the Chebyshev inequality we have
\begin{equation}
  \sup_k \mathscr{H}^{n-1}(J_{u_{i_k}^\epsilon})
  \leq \sup_k \frac{1}{\epsilon}\int_{J_{u_{i_k}}}\big(|u_{i_k}^+|+|u_{i_k}^-|\big)\, d\mathscr{H}^{n-1}\leq \frac{C}{\epsilon}, 
  \label{eqn:SBV3}
\end{equation}
%For the jump part of %$u_{i_k}^\varepsilon$, we can deduce that
%\begin{align*}
%  \sup_k\left\|D^j u_{i_k}^\varepsilon\right\|(D)&=\sup_k \int_{S(u_{i_k}^\varepsilon)}\left( (u_{i_k}^\varepsilon)^+-(u_{i_k}^\varepsilon)^- \right)\, d\H^{n-1}\\
%  &\leq \sup_k\int_{S(u_{i_k})}\left(|u_{i_k}^+|+|u_{i_k}^-|\right)\, d\H^{n-1}<\infty,
%
%
%\end{align*}
where we use that fact that $J_{u_{i_k}^\epsilon}\subset J_{u_{i_k}}\cap \left\{ u_{i_k}>\epsilon \right\}.$ 

Now from \eqref{eqn:SBV1}, \eqref{eqn:SBV2} and \eqref{eqn:SBV3}, 
we can apply \textcolor{red}{the} SBV compactness theorem (\cite[Theorem 4.7]{afp})
to conclude that $u^\epsilon\in SBV(D)$, and 
\begin{equation}
\begin{cases}
  \nabla u_{i_k}^\epsilon\rightharpoonup \nabla u^\epsilon\ \ \text{ in }\ \ L^1(D),\\
  \label{eqn:SBVconv1}
D^ju_{i_k}^\epsilon \overset{\star}\rightharpoonup D^j u^\epsilon\ \ \text{ in }\ \ \mathcal{M}(D),
\end{cases}
\end{equation}
{where $D^j$ denotes the jump part of the distributional gradient $Du$. }
Moreover, 
\begin{equation}
  \int_{D}|\nabla u^\epsilon|^2\leq \liminf_{k\rightarrow \infty}\int_{D}|\nabla u_{i_k}^\epsilon|^2\leq\liminf_{k\rightarrow \infty}\int_D |\nabla u_{i_k}|^2.
  \label{eqn:epsilonlsc}
\end{equation}
Since $\nabla u^\epsilon=\nabla u\chi_{ \{u>\epsilon\}}\rightarrow \nabla u$ a.e. in $D$
as $\epsilon\to 0$, by Fatou's lemma we have that
\begin{equation}
  \int_{D}|\nabla u|^2\leq \liminf_{\epsilon\rightarrow0} \int_{D}|\nabla u^\epsilon|^2\leq \sup_{k}\int_{D}|\nabla u_{i_k}|^2<\infty,
  \label{}
\end{equation}
{and} this implies $\nabla u\in L^2(D)$. 
From the dominated convergence theorem we have that 
\begin{equation}
  \nabla u^\epsilon\rightarrow \nabla u \text{ in }L^2(D)\ \ {\rm{as}}\ \ \epsilon\to 0.
  \label{eqn:gradientL2conv}
\end{equation}
For the jump part of $u$, since $u\in BV(\R^n)$, we get
\begin{equation}
\int_{J_u}|u^+-u^-|\, d\mathscr{H}^{n-1}<\infty. 
\label{eqn:jumpbound}
\end{equation}
Notice that
\begin{equation}
  D^j u^\epsilon=\left( (u^\epsilon)^+-(u^\epsilon)^- \right)\nu_u\mathscr{H}^{n-1}\lfloor_{J_u}.
  \label{eqn:jumpue}
\end{equation}
By \eqref{eqn:jumpbound}, \eqref{eqn:jumpue} and the dominated convergence theorem, we have
\begin{equation}
  D^j u^\varepsilon\rightarrow D^j u \text{ in } \mathcal{M}(D)\ \ {\rm{as}}\ \ \epsilon\to 0.
  \label{eqn:jumpconv}
\end{equation}
Since both convergence of \eqref{eqn:gradientL2conv} and \eqref{eqn:jumpconv} are strong, the Cantor part $D^c u$ of $Du$ vanishes.  \textcolor{red}{In} fact, for any open set $A$, 
\begin{equation}
  \begin{split}
    |Du|(A)&\leq \liminf_{\epsilon\rightarrow 0}|Du^\epsilon|(A)\\
    &=\liminf_{\epsilon\rightarrow0}\Big( \int_{A}|\nabla u^\epsilon|\,dx +|D^ju^\epsilon|(A)\Big)\\
    &=\int_{A}|\nabla u|\,dx+|D^ju|(A)\textcolor{red}{,}
  \end{split}
  \label{}
\end{equation}
which implies $|D^c u|(A)=0$. Hence $D^c u\equiv 0$ and $u\in SBV(\R^n)$. 
From \eqref{eqn:L1conv} we can derive that $|\supp u\setminus D|=0$, and $|\{u>0\}|=V_0$.

\medskip
\noindent{\it Claim 3}: The lower semicontinuity property holds for functional $\mathcal{J}$.
From \eqref{eqn:epsilonlsc} and \eqref{eqn:gradientL2conv}, we can conclude that 
\begin{equation}
  \int_{D}|\nabla u|^2\leq \lim_{k\rightarrow\infty}\int_{D}|\nabla u_{i_k}|^2.
  \label{eqn:gradientlsc}
\end{equation}
For any open set $A\subset \R^n$, in view of the bound estimate \eqref{eqn:SBV2}, we can apply the lower semicontinuity result in (\cite[Theorem 2.12]{B}) to $\left\{ u_{i_k}^\epsilon \right\}$ to obtain
\begin{equation}
  \int_{J_{u^\epsilon}\cap A}\left(|(u^\epsilon)^+|+|(u^\epsilon)^-|\right)\, d\mathscr{H}^{n-1}
  \leq \liminf_{k\rightarrow\infty}\int_{J_{u_{i_k}^\epsilon}\cap A}\big(|(u_{i_k}^\epsilon)^+|+|(u_{i_k}^\epsilon)^-|\big)\, d\mathscr{H}^{n-1}.
  \label{eqn:lowersemi2}
\end{equation}
Passing the $\epsilon$ to $0$ and applying the monotone convergence theorem to the left hand side of \eqref{eqn:lowersemi2} gives
\begin{equation}
  \int_{J_u\cap A}\left( |u^+|+|u^-| \right)\, d\mathscr{H}^{n-1}
  \leq \liminf_{k\to\infty}\int_{J_{u_{i_k}}\cap A}\big( |u_{i_k}^+|+|u_{i_k}^-| \big)\, d\mathscr{H}^{n-1}. 
  \label{}
\end{equation}
Choose $A=\R^n\setminus D$, we then get $\mathscr{H}^{n-1}(J_u\setminus D)=0$ and hence $u\in \mathcal{S}$. 
From \eqref{eqn:L1conv}, \eqref{eqn:gradientlsc}, and \eqref{eqn:lowersemi2}, we can conclude
that
\begin{equation}
  \mathcal{J}(u)\leq \liminf_k \mathcal{J}(u_{i_k})=\inf_{u\in \mathcal{S}} \mathcal{J}(u)
  \label{}
\end{equation}
which entails $u$ is a minimizer of the problem. 
\end{proof}

\section{Some properties on smooth critical points}
\setcounter{equation}{0}
\setcounter{theorem}{0}

%In this section, we will show that for $m>0$, any critical point $(u, \Omega)$ of the functional $\mathbb{J}_{m,1}$ is radially symmetric.

In this section, we will show that smooth solutions are stationary critical points.

\smallskip
For a  bounded  $C^{2}$-domain $\Omega\subset\mathbb R^n$,  since $\textcolor{red}{\mathcal{J}_m}(\cdot,\Omega):H^1(\Omega)\mapsto\mathbb R$ is convex,
it is readily seen  \textcolor{red}{in} \cite{BBN} that there exist\textcolor{red}{s} a unique critical point, denoted
as $u_{\Omega}$, of 
\begin{equation}\label{min2.1}
\mathcal{J}_{m}(v, \Omega):=\frac12\int_\Omega |\nabla v|^2\,dx+\frac{1}{2m}\big(\int_{\partial\Omega} |v|\,d\sigma\big)^2-\int_\Omega v\,dx,
\end{equation}
over $v\in H^1(\Omega)$. In fact, $u_\Omega$ is a minimal point of ${\mathcal{J}_m}(\cdot,\Omega)$ over $v\in H^1(\Omega)$. Since $\mathcal{J}_m(|u_\Omega|,\Omega)\le \mathcal{J}_m(u_\Omega,\Omega)$,
we conclude that $u_\Omega\ge 0$. Moreover,  we have the following proposition on the regularity of $\overline\Omega$. 
\begin{proposition}\label{regularity} If $\Omega\subset\mathbb R^n$ is a $C^2$ bounded domain, and $u\in H^1(\Omega)$ is a minimizer of $\mathcal{J}_m(\cdot,\Omega)$ over
$H^1(\Omega)$, then $u\in W^{1,p}(\Omega)$ for any $1\le p<\infty$ and 
\begin{equation}\label{1p-regularity} \max\big\{\|u\|_{W^{1,p}(\Omega)}, \|(\nabla u)^*\|_{L^p(\partial\Omega)}\big\}\le C(m, p, \|\Omega\|_{C^2}).
\end{equation}
\end{proposition}
\begin{proof}  For any $\epsilon>0$, consider $\mathcal{J}_m^\epsilon(\cdot, \Omega)$, an $\epsilon$-regularization of $\mathcal{J}_m(\cdot,\Omega)$, which is defined by
$$
\mathcal{J}_{m}^{\epsilon}(v,\Omega)=\frac12\int_\Omega |\nabla v|^2\,dx+\frac{1}{2m}\big(\int_{\partial\Omega} \sqrt{v^2+\epsilon^2}\,d\sigma\big)^2-\int_\Omega v\,dx.
$$
Let $v_\epsilon\in H^1(\Omega)$ be a minimizer of $\mathcal{J}_m^\epsilon(\cdot,\Omega)$, whose existence is standard. Then $v_\epsilon\ge 0$ in $\Omega$,
and direct calculations imply that
$v_\epsilon$ is a weak solution to the following Neumann boundary value problem:
$$\begin{cases}-\Delta v_\epsilon=1 & \ {\rm{in}}\ \Omega,\\
\displaystyle\frac{\partial v_\epsilon}{\partial\nu}=g_\epsilon:=\big(\frac{1}{m}\int_{\partial\Omega}\sqrt{v_\epsilon^2+\epsilon^2}\,d\sigma\big)\frac{v_\epsilon}{\sqrt{v_\epsilon^2+\epsilon^2}} & 
\ {\rm{on}}\ \partial\Omega.
\end{cases}
$$
It is easy see that 
$$\mathcal{J}_m^\epsilon(v_\epsilon,\Omega)\le \mathcal{J}_m^\epsilon(1,\Omega)\le C(m, |\partial\Omega|, |\Omega|), \ \forall 0<\epsilon\le 1.$$
This, combined with the Poincar\'e inequality, implies that
$$\int_\Omega |\nabla v_\epsilon|^2+\big(\int_{\partial\Omega}|v_\epsilon|\big)^2\le C(m, |\partial\Omega|, |\Omega|), \ \forall 0<\epsilon\le 1, $$
and hence
$$\|v_\epsilon\|_{H^1(\Omega)}\le C(m, |\partial\Omega|, |\Omega|), \ \forall 0<\epsilon\le 1.$$
Since $\displaystyle |g_\epsilon|\le \frac{1}{m}\int_{\partial\Omega} \sqrt{1+v_\epsilon^2}$ on $\partial\Omega$, this implies that $g_\epsilon\in L^\infty(\partial\Omega)$, and
$$\|g_\epsilon\|_{L^\infty(\partial\Omega)}\le C(m, |\partial\Omega|, |\Omega|), \ \forall 0<\epsilon\le 1.$$
Therefore we can apply the standard elliptic theory to conclude that  $v_\epsilon\in W^{1,p}(\Omega)$ for any $1\le p<\infty$, and
$$
\|v_\epsilon\|_{W^{1,p}(\Omega)}\le C(m, p, \|\Omega\|_{C^2}), \ \forall 0<\epsilon\le 1.$$
In fact, we have the stronger estimate, namely the $L^p$-norm of the non-tangential maximal function of $\nabla v_\epsilon$ can be 
bounded that of $g_\epsilon$, i.e.
\begin{equation}\label{non-tangential-estimate}
\big\|(\nabla v_\epsilon)^*\big\|_{L^p(\partial\Omega)}\le C(m, p, |\Omega\|_{C^2})\|g_\epsilon\|_{L^p(\partial\Omega)}, \ \forall 1<p<\infty.
\end{equation}
Hence we may assume, after taking a possible subsequence, that there exists $v\in W^{1,p}(\Omega)$, $p\in (1,\infty)$, such that
$$v_\epsilon\rightharpoonup v \ {\rm{in}}\ W^{1,p}(\Omega), \ \forall 1\le p<\infty.$$
Now we want to show that $v$ is also a minimizer of $\mathcal{J}_m(\cdot,\Omega)$. In fact, for any function $w\in H^1(\Omega)$ we have that
$$\mathcal{J}_m^\epsilon(v_\epsilon, \Omega)\le \mathcal{J}_m^\epsilon(w, \Omega).$$
Since $v_\epsilon\rightharpoonup v$ in $H^1(\Omega)$, it follows from the lower semicontinuity that
$$\mathcal{J}_m(v, \Omega)\le\liminf_{\epsilon\rightarrow 0}\mathcal{J}_m^\epsilon(v_\epsilon, \Omega)\le \liminf_{\epsilon\rightarrow 0} \mathcal{J}_m^\epsilon(w, \Omega)
=\mathcal{J}_m(w, \Omega).$$
Since $\mathcal{J}_m(\cdot,\Omega)$ is convex over $H^1(\Omega)$, there is a unique minimizer of $\mathcal{J}_m(\cdot,\Omega)$ in $H^1(\Omega)$.
Hence $u\equiv v$ in $\Omega$. This proves \eqref{1p-regularity}. 
\end{proof}

\bigskip
It follows from Proposition \ref{regularity} and the Sobolev embedding theorem 
that $u\in C^{\alpha}(\overline\Omega)$ for any $0<\alpha<1$. Hence, by direct calculations, we obtain that
$u=u_\Omega\ge 0$ is a weak \textcolor{red}{solution to} the following boundary value problem
\begin{equation}\label{critical2.1}
\begin{cases}
-\Delta u=1 & \ {\rm{in}}\ \Omega,\\
\displaystyle\frac{\partial u}{\partial\nu}=-\frac{1}{m}\int_{\partial\Omega} u\,d\sigma & \ {\rm{on}}\ \partial\Omega\cap\{x: u(x)>0\},\\
\displaystyle\frac{\partial u}{\partial\nu}\ge-\frac{1}{m}\int_{\partial\Omega} u\,d\sigma & \ {\rm{on}}\ \partial\Omega\cap\{x: u(x)=0\}.
\end{cases}
\end{equation}
It is readily seen that $u\not\equiv 0$ on $\partial\Omega$.  
The following lemma indicates that any nonnegative weak solution
of \eqref{critical2.1} also minimizes $\mathcal{J}_m(\cdot,\Omega)$.

\begin{lemma}\label{minimizer} For any bounded $C^2$-domain $\Omega\subset\mathbb R^n$, if $u\in H^1(\Omega)\cap C^{1}(\overline\Omega)$ is a nonnegative
weak solution of \eqref{critical2.1}, then 
\begin{equation}\label{minimizer0}
\mathcal{J}_m(u,\Omega)\le \mathcal{J}_{m}(v, \Omega), \ \forall\  v\in H^1(\Omega).
\end{equation}
\end{lemma}
\begin{proof} 
For any $v\in H^1(\Omega)$, multiplying \eqref{critical2.1} by $u-v$ and integrating over $\Omega$, we obtain
\begin{eqnarray}\label{minimizer00}
&&\int_\Omega |\nabla u|^2\,dx-\int_\Omega u\,dx-\int_{\partial\Omega}\frac{\partial u}{\partial\nu} u\,d\sigma\nonumber\\
&&=\int_\Omega \nabla u\cdot\nabla v\,dx-\int_\Omega v\,dx -\int_{\partial\Omega}\frac{\partial u}{\partial\nu} v\,d\sigma.
\end{eqnarray}
From \eqref{critical2.1}$_2$, we see that
\begin{eqnarray*}
-\int_{\partial\Omega}\frac{\partial u}{\partial\nu} u\,d\sigma=\big(\frac{1}{m}\int_{\partial\Omega} u\,d\sigma\big)\int_{\partial\Omega} u
=\frac{1}{m}\big(\int_{\partial\Omega} u\,d\sigma\big)^2.
\end{eqnarray*}
On the other hand, we have
\begin{eqnarray*}
&&-\int_{\partial\Omega}\frac{\partial u}{\partial\nu} v\,d\sigma=-\int_{\partial\Omega\cap\{u(x)>0\}}\frac{\partial u}{\partial\nu} v\,d\sigma
-\int_{\partial\Omega\cap\{u(x)=0\}}\frac{\partial u}{\partial\nu} v\,d\sigma\\
&&=\big(\frac{1}{m}\int_{\partial\Omega} u\,d\sigma\big)\int_{\partial\Omega\cap\{u(x)>0\}} v\,d\sigma
-\int_{\partial\Omega\cap\{u(x)=0\}}\frac{\partial u}{\partial\nu} v\,d\sigma\\
&&=\big(\frac{1}{m}\int_{\partial\Omega} u\,d\sigma\big)\int_{\partial\Omega} v\,d\sigma
-\int_{\partial\Omega\cap\{u(x)=0\}}\big(\frac{\partial u}{\partial\nu}+\frac{1}{m}\int_{\partial\Omega} u\,d\sigma\big)v \,d\sigma\\
&&=\big(\frac{1}{m}\int_{\partial\Omega} u\,d\sigma\big)\int_{\partial\Omega\cap\{v(x)>0\}} v\,d\sigma+
\big(\frac{1}{m}\int_{\partial\Omega} u\,d\sigma\big)\int_{\partial\Omega\cap\{v(x)\le 0\}} v\,d\sigma\\
&&\ \ -\big(\frac{1}{m}\int_{\partial\Omega} u\,d\sigma\big)\int_{\partial\Omega\cap\{u(x)=0\}\cap\{v(x)\le 0\}}v \,d\sigma
-\int_{\partial\Omega\cap\{u(x)=0\}\cap\{v(x)\le 0\}}
\frac{\partial u}{\partial\nu}v \,d\sigma\\
&&\ \ -\int_{\partial\Omega\cap\{u(x)=0\}\cap\{v(x)> 0\}}\big(\frac{\partial u}{\partial\nu}+\frac{1}{m}\int_{\partial\Omega} u\,d\sigma\big)v \,d\sigma\\
&&\le \big(\frac{1}{m}\int_{\partial\Omega} u\,d\sigma\big)\int_{\partial\Omega} |v|\,d\sigma-\int_{\partial\Omega\cap\{u(x)=0\}\cap\{v(x)\le 0\}}\frac{\partial u}{\partial\nu}v\,d\sigma\\
&&\ \ -\int_{\partial\Omega\cap\{u(x)=0\}\cap\{v(x)>0\}}\big(\frac{\partial u}{\partial\nu}+\frac{1}{m}\int_{\partial\Omega} u\,d\sigma\big)v \,d\sigma.
\end{eqnarray*}
It follows from \eqref{critical2.1}$_3$ that 
$${\big(\frac{\partial u}{\partial\nu}(x)+\frac{1}{m}\int_{\partial\Omega} u\,d\sigma\big)}v(x)\ge 0, \ \forall x\in \partial\Omega\cap\{u(x)=0\}\cap\{v(x)>0\},$$
and hence
$$\int_{\partial\Omega\cap\{u(x)=0\}\cap\{v(x)>0\}}\big(\frac{\partial u}{\partial\nu}+\frac{1}{m}\int_{\partial\Omega} u\,d\sigma\big)v \,d\sigma\ge 0.$$
Since $u\in C^1(\overline\Omega)$ satisfies $u>0$ in $\Omega$, it follows that $\frac{\partial u}{\partial\nu}(x)\le 0$ on
$\partial\Omega\cap\{u(x)=0\}$ and hence
$$\int_{\partial\Omega\cap\{u(x)=0\}\cap\{v(x)\le 0\}}\frac{\partial u}{\partial\nu}v\,d\sigma\ge 0.$$
Thus we obtain
\begin{eqnarray*}
-\int_{\partial\Omega}\frac{\partial u}{\partial\nu} v\,d\sigma\le \big(\frac{1}{m}\int_{\partial\Omega} u\,d\sigma\big)\int_{\partial\Omega} |v|\,d\sigma,
\end{eqnarray*}
and hence
\begin{eqnarray*}
&&\int_\Omega \nabla u\cdot\nabla v\,dx-\int_\Omega v\,dx -\int_{\partial\Omega}\frac{\partial u}{\partial\nu} v\,d\sigma\\
&&\le \int_\Omega \nabla u\cdot\nabla v\,dx-\int_\Omega v\,dx +\big(\frac{1}{m}\int_{\partial\Omega} u\,d\sigma\big)\int_{\partial\Omega} |v|\,d\sigma\\
&&\le \frac12\int_{\Omega}|\nabla u|^2\,dx+\frac12\int_{\Omega}|\nabla v|^2\,dx-\int_\Omega v\,dx 
+\frac{1}{2m}\big(\int_{\partial\Omega} u\,d\sigma\big)^2+\frac{1}{2m}\big(\int_{\partial\Omega} |v|\,d\sigma\big)^2.
\end{eqnarray*}
Substituting this into \eqref{minimizer00} yields that $\mathcal{J}_m(u,\Omega)\le \mathcal{J}_m(v,\Omega)$. 
\end{proof}

\bigskip
For $m>0$, it follows from the discussion above that if $u\in H^1(\Omega)$ is a critical point of $\mathcal{J}_m(\cdot,\Omega)$, then $u\ge 0$ in $\overline\Omega$. 
If, in addition, $u>0$ in $\overline\Omega$, then it follows from \eqref{critical2.1} that $u$ solves
\begin{equation}\label{critical2.2}
\begin{cases}
-\Delta u=1 & \ {\rm{in}}\ \Omega,\\
\displaystyle\frac{\partial u}{\partial\nu}=-\frac{1}{m}\int_{\partial\Omega} u\,d\sigma & \ {\rm{on}}\ \partial\Omega.
\end{cases}
\end{equation}
Thus it follows from the standard elliptic theory that $u\in C^{1,\beta}(\overline\Omega)$ for all $0<\beta<1$. 
However, the following example shows that there exists a bounded $C^2$-domains $\Omega$ such that any minimizer
$u\in H^1(\Omega)$ to $\mathcal{J}_m(\cdot,\Omega)$ has zero points on $\partial\Omega$.

\begin{example} For $n=2$ and $\Omega=\{x\in\mathbb R^2: \ 1<|x|<2\}$. If $0<m<3\pi-4\pi\ln 2$,
then $u(x)=-\frac14|x|^2+c_1\ln |x|+c_2$ for $x\in \Omega$, with
$$c_1=\frac{m+3\pi}{2m+4\pi \ln 2}, \ c_2=\frac{2m-(m-\pi)\ln 2}{2m +4\pi \ln 2},$$
is the unique minimizer of $\mathcal{J}_m(\cdot,\Omega)$ over $H^1(\Omega)$. 
\end{example}

\begin{proof} Notice that $\partial \Omega=\partial B_1\cup\partial B_2$.
It is easy to see that $u>0$ in $\Omega\cup \partial B_1$ and $u=0$ on $\partial B_2$,
and satisfies
\begin{equation}
\begin{cases}
-\Delta u= 1& \ {\rm{in}}\ \Omega,\\
\displaystyle\frac{\partial u}{\partial\nu}=-\frac{1}{m}\int_{\partial B_1}u& \ {\rm{on}}\ \partial B_1,\\
\displaystyle\frac{\partial u}{\partial\nu}>-\frac{1}{m}\int_{\partial \Omega}u& \ {\rm{on}}\ \partial B_2.
\end{cases}
\end{equation}
From Lemma \ref{minimizer}, $u$ is a minimizer of $\mathcal{J}_m(\cdot, \Omega)$ in $H^1(\Omega)$. 
\end{proof}

\begin{proposition} \label{stationarity} If $u\in W^{2,2}(\Omega)$ is a critical point of $\mathcal{J}_m(\cdot,\Omega)$, then it is also critical with respect to the domain variation, i.e.,
\begin{equation}\label{stationarity1}
\frac{d}{dt}\big|_{t=0}\mathcal{J}_m(u^t,\Omega)=0, 
\end{equation}
where $u^t(x)=u(F(t,x))$, and $F(\cdot,\cdot):(-\delta, \delta)\times\overline\Omega\mapsto\overline\Omega$ is a $C^1$-family of $C^2$-diffeomorphism satisfying
$$
\begin{cases} F(0, x)=x, \ \forall x\in\overline\Omega, \\
F(t,x)\in \partial\Omega, \ \forall (x,t)\in\partial\Omega\times (-\delta,\delta).
\end{cases}
$$
\end{proposition}
\begin{proof} Define the deformation vector field $\eta(x)=\frac{d}{dt}|_{t=0}F(t,x)$ for $x\in\overline\Omega$. 
Then
$$\eta(x)\in T_x(\partial\Omega) \ {\rm{or}}\ \eta(x)\cdot\nu(x)=0, \ \forall x\in\partial\Omega.$$
By direct calculations, we have
%\footnote{To make the argument rigorous,we need to assume $u\in W^{2,q}(\Omega)$ for some $1<q<\infty$.}
\begin{eqnarray*}
&&\frac{d}{dt}\big|_{t=0}\big(\frac12 \int_\Omega |\nabla u^t|^2\,dx\big)\\
&=&-\frac12\int_{\Omega}|\nabla u|^2 {\rm{div}}\eta\,dx+\int_\Omega u_i u_j \eta^i_j\,dx\\
&=&-\frac12\int_{\Omega}|\nabla u|^2 {\rm{div}}\eta\,dx+\int_{\partial\Omega}\eta\cdot\nabla u \frac{\partial u}{\partial\nu}\,d\sigma-\int_\Omega \Delta u (\eta\cdot\nabla u)\,dx
-\frac12\int_\Omega \eta\cdot \nabla(|\nabla u|^2)\,dx\\
&=&-\frac12\int_{\Omega}{\rm{div}}(|\nabla u|^2\eta)\,dx+\int_{\partial\Omega}\eta\cdot\nabla u \frac{\partial u}{\partial\nu}\,d\sigma-\int_\Omega \Delta u (\eta\cdot\nabla u)\,dx\\
&=&-\frac12\int_{\partial\Omega}|\nabla u|^2\eta\cdot\nu\,d\sigma+\int_{\Omega}\eta\cdot\nabla u\,dx+\int_{\partial\Omega}\eta\cdot\nabla u \frac{\partial u}{\partial\nu}\,d\sigma\\
&=&\int_{\Omega}\eta\cdot\nabla u\,dx+\int_{\partial\Omega}\eta\cdot\nabla_{\rm{tan}} u \frac{\partial u}{\partial\nu}\,d\sigma,
\end{eqnarray*}
where we have used the equation \eqref{critical2.1}$_1$, and $\nabla_{\rm{tan}} u=(\mathbb{I}_n-\nu\otimes\nu)\nabla u$.
\begin{eqnarray*}
\frac{d}{dt}\big|_{t=0}\big\{\frac1{2m} (\int_{\partial\Omega} u^t\,d\sigma)^2\big\}
=\frac1{m} \int_{\partial\Omega} u\,d\sigma\int_{\partial\Omega}\eta\cdot\nabla_{\rm{tan}} u\,d\sigma.
\end{eqnarray*}
It is readily seen that
\begin{eqnarray*}
&&\frac{d}{dt}\big|_{t=0}\big(-\int_\Omega u^t\,dx\big)=-\int_\Omega \eta \cdot\nabla u\,dx.
\end{eqnarray*}
Putting these identities together, we obtain that
\begin{eqnarray*}
\frac{d}{dt}\big|_{t=0}\mathcal{J}_m(u^t,\Omega)
&=&\int_{\partial\Omega}\eta\cdot\nabla_{\rm{tan}} u\big(\frac{\partial u}{\partial\nu}+\frac1{m} \int_{\partial\Omega} u\,d\sigma\big)\,d\sigma\\
&=&\int_{\partial\Omega\cap\{u>0\}}\eta\cdot\nabla_{\rm{tan}} u\big(\frac{\partial u}{\partial\nu}+\frac1{m} \int_{\partial\Omega} u\,d\sigma\big)\,d\sigma=0.
\end{eqnarray*}
This completes the proof. 
\end{proof}

\begin{definition}\label{critical} Given a bounded $C^{2}$-domain $\Omega\subset\mathbb R^n$, let $u=u_\Omega\in H^1(\Omega)$
be the unique minimizer of \eqref{min2.1}. We say that $(u, \Omega)$ is a critical point of $\mathcal{J}_{m}(\cdot,\cdot)$, if either
$I(t)=\mathcal{J}_m(u_{\Omega(t)}, \Omega(t))$ is not differentiable at $t=0$, or 
\begin{equation}\label{critical2.3}
\frac{d}{dt}\big|_{t=0} \mathcal{J}_m(u_{\Omega(t)}, \Omega(t))=0,
\end{equation}    
where $\Omega(t)=\{F(t,x): x\in\Omega\}$ and $u_{\Omega(t)}$ is the unique minimizer of $\mathcal{J}_m(\cdot, \Omega(t))$ over $H^1(\Omega(t))$.
Here $F(t, x): (-\delta, \delta)\times\overline\Omega\mapsto \mathbb R^n$ is any $C^1$-family of $C^{2}$-volume preserving diffeomorphism, 
that is generated by a vector field $\eta\in C^2(\overline\Omega,\R^n)$, i.e.,
$$\frac{dF}{dt}(t,x)=\eta(F(t,x)); \ F(0,x)=x,\ \forall x\in\Omega, \ -\delta<t<\delta.$$ 
\end{definition}
Here $(u_{\Omega(0)}, \Omega(0))=(u,\Omega)$.

\medskip
Now we have
\begin{theorem}\label{uniqueness} For $m>0$ and a bounded $C^{2}$-domain $\Omega\subset\mathbb R^n$, let $u_{\Omega}$ be the unique minimizer
of $\mathcal{J}_m(\cdot,\Omega)$ over $H^1(\Omega)$. If  $u_\Omega$ is positive in $\overline\Omega$, then $(u_\Omega, \Omega)$ is a critical point of $\mathcal{J}_m(\cdot,\cdot)$ 
if and only if the following identity holds:
\begin{equation}\label{stationarity2}
\frac{1}{2}|\nabla_{\rm{tan}}u_{\Omega}|^2 -u_\Omega-\frac12\big(\frac{1}{m}\int_{\partial \Omega} u_{\Omega}\big)^2 +\big(\frac{1}{m}\int_{\partial \Omega} u_{\Omega}\big) u_{\Omega}H \equiv 
{\rm{constant}}, \ {\rm{on}}\ \partial\Omega,
\end{equation}
where $H$ denotes the mean curvature of $\partial\Omega$. In particular, for any ball $B_R\subset\mathbb R^n$
with radius $R$, $(u_{B_R}, B_R)$ is a critical point of $\mathcal{J}_m(\cdot,\cdot)$.
\end{theorem}

\begin{proof} For simplicity, denote $u=u_\Omega$. Since $u\in C(\overline\Omega)$ is positive, it follows that $u$ solves \eqref{critical2.3} 
so that $u\in C^{1,\alpha}(\overline\Omega)\cap W^{2,2}(\Omega)$. Hence there exists $\delta_0>0$ such that $u\ge \delta_0$ in $\overline\Omega$.
For a small $0<\delta_1<<\delta_0$ and an open set $U\supset\overline\Omega$, let $F(t,x): (-\delta_1,\delta_1)\times U\mapsto\mathbb R^n$ be a $C^1$-family of $C^{2}$-volume preserving
diffeomorphism, generated by a vector field $\eta\in C^{2}(U, \mathbb R^n)$. 
It is readily seen that 
$\Omega(t)=F(t)(\Omega)$, $-\delta_1<t<\delta_1$, is a $C^1$-family of  bounded $C^{2}$-domains. By an argument similar to that of Proposition {\ref{regularity}}, we can show
that $u(t)\equiv u_{\Omega(t)}(F(t,\cdot))\rightarrow u$ in $C^0(\overline\Omega)$ as $t\rightarrow 0$
so that there exists $0<\delta_2<\delta_1$ such that
$u(t)(y)\ge \frac{\delta_0}2$ for $y\in \overline{\Omega(t)}$ and $t\in (-\delta_2,\delta_2)$. Hence $u(t)$, $-\delta_2<t<\delta_2$, solves
\begin{equation}\label{critical2.4}
\begin{cases}
-\Delta u(t)=1 & \ {\rm{in}}\ \Omega(t),\\
\displaystyle\frac{\partial}{\partial\nu} u(t)=-\frac{1}{m}\int_{\partial\Omega(t)} u(t)(y)\,d\sigma & \ {\rm{on}}\ \partial\Omega(t).
\end{cases}
\end{equation}
Applying Proposition {\ref{regularity}} again, we have that for any $1<p<\infty$,
$$\big\|u(t)\big\|_{W^{2,2}(\Omega(t))}+\big\|u(t)\big\|_{W^{1,p}(\Omega(t))}\le C(p), \ \ t\in (-\delta_2,\delta_2).$$
This implies $\mathcal{J}_m(u(t),\Omega(t))\in C^1((-\delta_2, \delta_2))$.

It follows from $|\Omega(t)|=|\Omega|$ for $-\delta_2<t<\delta_2$ that 
\begin{equation}\label{divergence-free}
\int_\Omega {\rm{div}}\eta=0.
\end{equation}
Now we calculate $\frac{d}{dt}\mathcal{J}_m(u(t),\Omega(t))$ for $t\in (-\delta_2,\delta_2)$. We claim that
\begin{eqnarray}\label{first-derivative0}
\frac{d}{dt}\mathcal{J}_m(u(t),\Omega(t))
&=& \int_{\partial \Omega(t)}\big[\frac{1}{2}|\nabla_{\rm{tan}}u(t)|^2 -\frac{1}{2}|\nabla_{\nu}  u(t)|^2 -u(t)\nonumber\\
&&\qquad\qquad+\big(\frac{1}{m}\int_{\partial \Omega(t)} u(t)\big) u(t)H(t) \big]\eta \cdot \nu\,d\sigma,
\end{eqnarray}
for all $t\in (-\delta_2,\delta_2)$. Here $H(t)$ denotes the mean curvature of $\partial\Omega(t)$,
and $\nabla_{\rm{tan}}f=(\mathbb{I}_n-\nu\otimes\nu)\nabla f$ denotes the tangential derivative of $f$ on $\partial\Omega(t)$.

To simplify the proof, denote $u(t,x)=u(t)(x)$ and set $v(t,x)\equiv\frac{\partial}{\partial t} u(t,x)$, $x\in\Omega(t)$.  Notice that
$\Omega=\Omega(0)$ and $u_\Omega(x)=u(0,x)$, $x\in\Omega$. Recall the formula \cite[Corollary 5.2.8]{Henrot}
\begin{equation}\label{volume-derivative}
\frac{d}{dt}\int_{\Omega(t)} f(t,y)\,dy=\int_{\Omega(t)} \frac{\partial f}{\partial t}(t,y)\,dy+\int_{\partial\Omega(t)} f(t,y)\eta(y)\cdot\nu(t,y)\,d\sigma,
\end{equation}
for any $f\in C^1(\big\{(t,x): \ t\in (-\delta_2,\delta_2), \ x\in \Omega(t)\big\})$, where $\nu(t,\cdot)$ denotes the outward unit normal of $\partial\Omega(t)$.
Applying \eqref{volume-derivative}, we can calculate
\begin{eqnarray*}
&&I_1(t)\equiv\frac{d}{dt}\int_{\Omega(t)} \frac{1}{2}|\nabla u(t,x)|^2\,dx\\
&&=\int_{\Omega(t)} \nabla u(t,x) \cdot\nabla v(t,x)\,dx+\int_{\partial \Omega(t)} \frac{1}{2}|\nabla u(t,x)|^2 \eta(x) \cdot \nu(t,x) \,d\sigma\\
&&=-\int_{\Omega(t)}\Delta u(t,x)v(t,x)\,dx+\int_{\partial \Omega(t)} v(t,x) \partial_{\nu} u(t,x)\,d\sigma\\
&&\ \ +\int_{\partial \Omega(t)}\frac{1}{2}|\nabla u(t,x)|^2 \eta(x) \cdot \nu(t,x) \,d\sigma,
\end{eqnarray*}
and
\begin{eqnarray*}
I_3(t)\equiv\frac{d}{dt}\int_{\Omega(t)} u(t,x)\,dx
=\int_{\Omega(t)} v(t,x)\,dx+\int_{\partial \Omega(t)} u(t,x)\eta(x) \cdot\nu(t,x)\,d\sigma.
\end{eqnarray*}
Also recall the formula \cite[Proposition 5.4.18]{Henrot}
\begin{eqnarray}\label{surface-derivative}
\frac{d}{dt}\int_{\partial\Omega(t)} f(t,x)\,dx&=&\int_{\partial\Omega(t)} \big(\frac{\partial f}{\partial t}(t,x)+\frac{\partial f}{\partial\nu}(t,x)\eta(x)\cdot\nu(t,x)\big)\,d\sigma\nonumber\\
&&\ +\int_{\partial\Omega(t)} f(t,x)H(t)(x)\eta(x)\cdot\nu(t,x)\,d\sigma,
\end{eqnarray}
for any $f\in C^1(\big\{(t,x): \ t\in (-\delta_2,\delta_2), \ x\in \Omega(t)\big\})$. 
%Here $\eta^{\rm{tan}}=(I_n-\nu\otimes\nu)\eta$ denotes the tangential component of $\eta$ on $\partial\Omega(t)$, and ${\rm{div}}_{\rm{tan}}\eta^{\rm{tan}}$
%denotes the tangential divergence of the tangential component of $\eta$. 
Applying \eqref{surface-derivative} and \eqref{critical2.4}, we find
\begin{eqnarray*}
&&I_2(t)\equiv\frac{d}{dt}\big\{\frac{1}{2m}\big(\int_{\partial \Omega_t} u(t,x)\,d\sigma\big)^2\big\}\\
&&=\big(\frac{1}{m}\int_{\partial \Omega(t)}u(t,x)\,d\sigma\big)\int_{\partial {\Omega(t)}} \big(v(t,x)
+(\frac{\partial u}{\partial\nu}(t,x) +u(t,x) H(t,x)) \eta(x) \cdot \nu(t,x)\big) \,d\sigma\\
&&=-\int_{\partial {\Omega(t)}} \frac{\partial u}{\partial\nu}(t,x)\big[v(t,x)
+(\frac{\partial u}{\partial\nu}(t,x) +u(t,x) H(t,x)) \eta(x) \cdot \nu(t,x)\big] \,d\sigma,
\end{eqnarray*}
where $H(t,x)=H(t)(x)$ denotes the mean curvature of $\partial\Omega(t)$ at $x\in \partial\Omega(t)$.

Adding $I_1(t)$, $I_2(t)$, and $-I_3(t)$ together, and applying the equation \eqref{critical2.4}$_1$ we obtain that
\begin{eqnarray}\label{critical2.5}
&&\frac{d}{dt}\mathcal{J}_m(u(t),\Omega(t))=I_1(t)+I_2(t)-I_3(t)\\
&&=\int_{\Omega(t)}(-\Delta u(t,x)-1)v(t,x)\,dx\nonumber\\
&&+\int_{\partial \Omega(t)}\big(\frac{1}{2}|\nabla u(t,x)|^2-|\frac{\partial u}{\partial\nu}|^2(t,x)+u(t,x)-\frac{\partial u}{\partial\nu}(t,x) u(t,x)H(t,x)\big)\eta(x) \cdot \nu(t,x) \,d\sigma\nonumber\\
&&=\int_{\partial \Omega(t)}\big(\frac{1}{2}|\nabla_{\rm{tan}} u(t,x)|^2-\frac12|\frac{\partial u}{\partial\nu}|^2(t,x)+u(t,x)-\frac{\partial u}{\partial\nu}(t,x) u(t,x)H(t,x)\big)\eta(x) \cdot \nu(t,x) \,d\sigma.
\nonumber
\end{eqnarray}
Thus, by setting $t=0$ and applying \eqref{critical2.4}$_2$, we obtain that
\begin{eqnarray}\label{critical2.6}
&&\frac{d}{dt}\big|_{t=0}\mathcal{J}_m(u(t),\Omega(t))\nonumber\\
&&=\int_{\partial \Omega}\Big(\frac{1}{2}|\nabla_{\rm{tan}} u|^2-\frac12|\frac{\partial u}{\partial\nu}|^2+u-(\frac{1}{m}\int_{\partial\Omega} u\,d\sigma)
uH\Big)\eta(x) \cdot \nu\,d\sigma.
\end{eqnarray}
Notice that for any given $C^1$-family of volume preserving $C^2$-diffeomorphism maps $F(t,x):(-\delta_1,\delta_1) \times \overline \Omega \mapsto \mathbb{R}^n$ for some $\delta_1>0$, it is necessary that the velocity field $\eta$ satisfies $\int_{\partial \Omega} \eta \cdot \nu d\sigma=0$. 
Substituting such {an} $\eta$ into \eqref{critical2.6}, we see that \eqref{stationarity2} holds
iff $(u_\Omega,\Omega)$ is a critical point of $\mathcal{J}_m(\cdot,\cdot)$.

Recall that when $\Omega=B_R$, the unique critical point of $\mathcal{J}_m(\cdot, B_R)$ is given by
\begin{equation}\label{ball-solution}
u_{B_R}(x)=\frac{R^2-|x|^2}{2n}+\frac{m}{n^2\omega_n R^{n-2}}, \ x\in B_R,
\end{equation}
where $\omega_n$ is the volume of the unit ball  in $\mathbb R^n$.
Since $u_{B_R}$ is smooth and positive in $\overline {B_R}$, and satisfies \eqref{stationarity2}, it follows that $(u_{B_R}, B_R)$ is a critical point of $\mathcal{J}_m(\cdot,\cdot)$. 
\end{proof}

\bigskip

\section{Stability of $(u_{B_R}, B_R)$}
\setcounter{equation}{0}
\setcounter{theorem}{0}

It follows from Theorem {\ref{uniqueness}} that for any $R>0$, $(u_{B_R}, B_R)$ is a critical point for $\mathcal{J}_m(\cdot,\cdot)$ for any $m>0$.
In this section, we will prove Theorem \ref{new3}, namely, {$(u_{B_R},B_R)$ is a} stable critical point
of $\mathcal{J}_m(\cdot,\cdot)$.

\begin{proof}[Proof of Theorem \ref{new3}]
It follows from the discussion in the previous section  that there exists $\delta_0>0$ such that $u(t,x)=u_{\Omega(t)}(x)$ is positive, satisfies
\eqref{critical2.4},  and is smooth in $\overline\Omega(t)$ for $t\in (-\delta,\delta)$. Hence by the formula \eqref{critical2.5} we have that for $t\in (-\delta,\delta)$, 
\begin{eqnarray}\label{critical3.1}
&&\frac{d}{dt}\mathcal{J}_m(u(t), \Omega(t))\nonumber\\
&&=\int_{\partial \Omega(t)}\left[\frac{1}{2}|\nabla u|^2(t,x) -|\frac{\partial u}{\partial\nu}|^2(t,x) -u(t,x)-\frac{\partial u}{\partial\nu}(t,x) u(t,x)H(t,x) \right]
\eta(x) \cdot \nu(t,x)\,d\sigma\nonumber\\
&&=I(t)+II(t)+III(t)+IV(t).
\end{eqnarray}
%Throughout this section we consider normal variations of $B_R$. Since $\nabla u_{B_R} \perp \partial B_R$ and $\eta \perp \partial B_R$, we have
%\begin{align*}
% &\frac{d}{dt}\big|_{t=0} \left[\left(\frac{1}{m}\int_{\partial {\Omega_t}} u_{\Omega_t}\right) \int_{\partial \Omega_t} (\eta \cdot e_i)(\nabla u_{\Omega_t} \cdot e_i)\right]\\
% =&\frac{d}{dt}\big|_{t=0} \left(\frac{1}{m}\int_{\partial {\Omega_t}} u_{\Omega_t}\right)\int_{\partial B_R}(\eta \cdot e_i(0))(\nabla u_{B_R} \cdot e_i(0))\\
% &+\left[\left(\frac{1}{m}\int_{\partial B_R} u_{B_R}\right) \int_{\partial B_R} \frac{d}{dt}\big|_{t=0}(\eta \cdot e_i)(\nabla u_{B_R} \cdot e_i(0))+(\eta \cdot e_i(0)) \frac{d}{dt}\big|_{t=0}(\nabla u_{\Omega_t} \cdot e_i) \right]\\
% &+\left[\left(\frac{1}{m}\int_{\partial B_R} u_{B_R}\right) \int_{\partial B_R} (\eta \cdot e_i(0))(\nabla u_{B_R} \cdot e_i(0)) (\eta \cdot \nu) H\right]\\ =&0.
%\end{align*} Therefore, 
%\begin{align} I''(0)=\frac{d}{dt}\big|_{t=0} \int_{\partial \Omega_t}\left(\frac{1}{2}|\nabla u_{\Omega_t}|^2  -u_{\Omega_t}+\left(\frac{1}{m}\int_{\partial \Omega_t} u_{\Omega_t}\right) u_{\Omega_t}H-|\nabla_{\nu}  u_{\Omega_t}|^2 \right)\eta \cdot \nu .\end{align}
To simplify the presentation, set 
$$v(x)=\frac{\partial u}{\partial t}(0,x), \ \ u_0(x)=u(0,x), \ x\in B_R,$$
and $\zeta(x)=\eta(x)\cdot\nu(x)$ for $x\in\partial B_R$. 
From the volume constraint $|\Omega(t)|=|B_R|$ for $t\in (-\delta, \delta)$, we claim that 
\begin{equation}\label{ling01}
\int_{\partial B_R} \zeta(x)\,d\sigma=\int_{B_R} {\rm{div}}\eta(x) \,dx=0,
\end{equation}
and
\begin{equation}
\label{ling1}
\int_{\partial B_R} \zeta(x) {\rm{div}} \eta (x)\,d\sigma= \int_{B_R} {\rm{div}} ({{\rm{div}}\eta\, \eta})\,dx=0.
\end{equation}
To see this, notice that since $\displaystyle|\Omega(t)|=\int_{B_R} JF(t,x)\,dx$ is constant, we have that
$$\frac{d}{dt}\big|_{t=0} \int_{B_R} JF(t,x)\,dx=\frac{d^2}{dt^2}\big|_{t=0} \int_{B_R} JF(t,x)\,dx=0.$$
While by direct calculations we have
$$ \frac{d}{dt} JF(t,x)=({\rm{div}}\eta\circ F(t,x)) JF(t,x), $$
and
$$
\frac{d^2}{dt^2} JF(t,x)=({\rm{div}}\eta\circ F(t,x))^2+(\nabla{\rm{div}}\eta\circ F(t,x)) (\eta\circ F(t,x)) JF(t,x).
$$
Thus we obtain 
\begin{equation*}
\begin{cases}
\displaystyle\int_{B_R} {\rm{div}}\eta(x)\,dx=0,\\
\displaystyle\int_{B_R} {\rm{div}}({{\rm{div}}\eta\, \eta})(x)\,dx=\int_{B_R} \big(({\rm{div}}\eta)^2+\eta\nabla{\rm{div}}\eta\big)(x)\,dx=0
\end{cases}
\end{equation*}
so that \eqref{ling01} and \eqref{ling1} hold.

From \eqref{ball-solution}, we see that
$$u_0= \frac{m}{n^2\omega_n R^{n-2}}\ {\rm{and}}\
\nabla u_0(x)=-\frac{x}{n}\ {\rm{on}}\ \partial B_R; \ \ \frac{\partial u_0}{\partial\nu}=-\frac{R}{n} \ {\rm{on}}\ \partial B_R.$$
Applying \eqref{surface-derivative}, we have
\begin{eqnarray}\label{ling0}
 \frac{d}{dt}\big |_{t=0} \big(\frac{1}{m}\int_{\partial \Omega(t)} u(t,x)\,d\sigma \big)
 &=&\frac{1}{m}\int_{\partial B_R} \big(v(x)+\frac{\partial u_0}{\partial \nu}(x) \zeta(x)
    +u_0(x)H(x) \zeta(x)\big)\,d\sigma\nonumber\\
 &=& \frac{1}{m}\int_{\partial B_R} v(x)\,d\sigma+\big(\frac{n-1}{n^2\omega_n R^{n-1}}-\frac{R}{nm}\big)\int_{\partial B_R}\zeta(x)\,d\sigma\nonumber\\
 &=&\frac{1}{m}\int_{\partial B_R} v(x)\,d\sigma,
\end{eqnarray}
where we have used $H=\frac{n-1}{R}$ on $\partial B_R$.

Now we want to show that $v$ solve the boundary value problem in $B_R$:
 \begin{equation}
	  \left\{
	    \begin{array}{ll}
	      -\Delta v=0, & \text{ in }B_R, \\
	     \displaystyle\ \ \frac{\partial v}{\partial \nu}  =\frac{\zeta}{n},&  \text{ on }\partial B_R.
	    \end{array}
	    \right.
	    \label{eqnvelocity}
\end{equation}
To see \eqref{eqnvelocity},  let $\phi\in C^\infty_0(B_{R+1})$.
Then by \eqref{volume-derivative} we have
\begin{eqnarray*}
0&=&\frac{d}{dt}|_{t=0}\int_{\Omega(t)} (\Delta u(t,x)+1)\phi(x)\,dx\\
&=&\int_{B_R} \Delta v(x)\phi(x)\,dx+\int_{\partial B_R} (\Delta u_0+1)\phi(x)\zeta(x)\,d\sigma\\
&=&\int_{B_R} \Delta v(x)\phi(x)\,dx,
\end{eqnarray*} 
where we have used the fact that $\Delta u_0+1=0$  on $\partial B_R$. Since $\phi$ is arbitrary,
we conclude that $\Delta v=0$ in $B_R$. 
To show $v$ satisfies the boundary condition \eqref{eqnvelocity}$_2$, we apply \eqref{ling0} and
\eqref{surface-derivative},
and proceed as follows.
\begin{eqnarray}\label{bdry2}
    &&0=\frac{d}{dt}\big|_{t=0}\int_{\partial \Omega(t)} \phi(x)\big[\nu(t,x)\cdot\nabla u(t,x)
    +(\frac{1}{m}\int_{\partial \Omega(t)}u(t,y)\,d\sigma)\big]\,d\sigma\nonumber\\
    &&= \int_{\partial B_R} \phi(x) \big(\frac{x}R\cdot\nabla v(x)+\frac{\partial \nu}{\partial t}(0,x)\cdot\nabla u_0 
    +[\frac{x}{R}\cdot\nabla(\frac{x}{|x|})\cdot\nabla u_0+\frac{x}{R}\otimes \frac{x}{R}:\nabla^2 u_0]\zeta(x)\big)\,d\sigma\nonumber\\
    &&\ + (\frac{1}{m}\int_{\partial B_R} v(x)\,d\sigma)\int_{\partial B_R}\phi(x)\,d\sigma
    +\int_{\partial B_R}\phi(x)\big(\frac{\partial u_0}{\partial\nu}+(\frac{1}{m}\int_{\partial B_R} u_0(x)\,d\sigma)\big)
    H(x)\zeta(x)\,d\sigma\nonumber\\
    &&= \int_{\partial B_R} \phi(x) \big(\frac{\partial v(x)}{\partial\nu} -\frac{1}{n}\zeta(x)
    +\frac{1}{m}\int_{\partial B_R} v(x)\,d\sigma \big),
\end{eqnarray}
where we have used the following facts:
$$\langle \frac{\partial\nu}{\partial t}(0,x),\nabla u_0(x)\rangle 
=-\frac{R}{n} \langle \frac{\partial\nu}{\partial t}(0,x),\nu(0,x)\rangle=0, \ {\rm{on}}\ \partial B_R, $$
$$\frac{x}{R}\cdot\nabla(\frac{x}{|x|})\cdot\nabla u_0=-\frac{1}{n}\frac{x}{R}\cdot\nabla(\frac{x}{|x|})\cdot x=0,
\ {\rm{on}}\ \partial B_R,$$
$$\frac{x}{R}\otimes \frac{x}{R}:\nabla^2 u_0=-\frac{1}{n} \frac{x}{R}\otimes \frac{x}{R}: I_n
=-\frac{1}{n}, \ {\rm{on}}\ \partial B_R,$$
and
$$
\frac{\partial u_0}{\partial\nu}+\frac{1}{m}\int_{\partial B_R} u_0(x)\,d\sigma=0, \ {\rm{on}}\ \partial B_R.
$$
It follows from \eqref{bdry2} that 
\begin{equation}\label{bdry1}
\frac{\partial v}{\partial\nu} =\frac{\zeta}{n}-\frac{1}{m}\int_{\partial B_R}v(x)\,d\sigma, \ {\rm{on}}\ \partial B_R.
\end{equation}
Since $\Delta v=0$ in $B_R$, we have
$$\int_{\partial B_R}\frac{\partial v}{\partial\nu}\,d\sigma=0, $$
this, combined with \eqref{bdry1} and $\displaystyle\int_{\partial B_R}\zeta=0$, implies  that 
\begin{equation}\label{zero-average}
\frac{1}{m}\int_{\partial B_R}v(x)\,d\sigma=0.
\end{equation}
Thus $v$ solves \eqref{eqnvelocity}. From \eqref{zero-average} and \eqref{ling0}, we also have that  
\begin{equation}
\label{ling}
\frac{d}{dt}\big |_{t=0} \big(\frac{1}{m}\int_{\partial \Omega (t)} u(t,x)\,d\sigma \big)=0.
\end{equation}

Next we want to compute the second order variation based on \eqref{critical3.1}. First, 
applying \eqref{surface-derivative}, we have
\begin{eqnarray}\label{I1}
    &&I'(0)=\frac{d}{dt}\big |_{t=0}\int_{\partial \Omega(t)} \frac{1}{2}|\nabla u|^2(t,x) \eta(x) \cdot \nu(t,x)\,d\sigma\nonumber \\
    &=& \int_{\partial B_R}\big(\nabla u_0(x) \cdot \nabla v(x) \zeta(x)+\frac12|\nabla u_0(x)|^2\eta(x)\cdot\frac{\partial \nu}{\partial t}(0,x)\big)\,d\sigma\nonumber\\
    &+&\int_{\partial B_R}\big(\eta(x)\cdot\nabla ^2 u_0(x) \cdot \nabla u_0(x)  \eta(x) \cdot \nu(x) + \frac{1}{2} |\nabla u_0(x)|^2 \eta\cdot\nabla (\eta(x)\cdot\nu(x))\big)\,d\sigma\nonumber\\
    &+&\int_{\partial B_R}\frac{1}{2} |\nabla u_0(x)|^2 H(x)(\eta(x) \cdot \nu(x))^2\,d\sigma,
\end{eqnarray}
where we have used the fact that $\nu(x)=\nu(0,x)$ for $x\in\partial B_R$.  

Since $\langle\frac{\partial \nu}{\partial t}(0,x), \nu(x)\rangle=0$ and $\eta(x)=\zeta(x) \nu (x)$ on $\partial B_R$, we see that
\begin{equation}\label{I1.1}
\int_{\partial B_R}\frac12|\nabla u_0(x)|^2\eta(x)\cdot\frac{\partial \nu}{\partial t}(0,x)\,d\sigma=0.
\end{equation}
Since $\nu(x)=\frac{x}{R}$ and $\nabla u_0(x)=-\frac{x}{n}$ on $\partial B_R$, by \eqref{eqnvelocity} we see that 
\begin{equation}\label{I1.2}
\int_{\partial B_R}\nabla u_0(x) \cdot \nabla v(x) \zeta(x)\,d\sigma=-\frac{R}{n^2}\int_{\partial B_R}\zeta^2(x)\,d\sigma.
\end{equation}
Direct calculations yield
\begin{eqnarray}\label{I1.3}
&&\int_{\partial B_R}\eta(x)\cdot\nabla ^2 u_0(x) \cdot \nabla u_0(x)  \eta(x) \cdot \nu(x)\,d\sigma\nonumber\\
&&=\int_{\partial B_R}\eta^i(x)(-\frac{|x|^2}{2n})_{ij} (-\frac{|x|^2}{2n})_{j}\zeta(x)\,d\sigma\nonumber\\
&&=\frac{1}{n^2}\int_{\partial B_R} (\eta(x)\cdot x)\zeta(x)\,d\sigma\nonumber\\
&&=\frac{R}{n^2}\int_{\partial B_R} \zeta^2(x)\,d\sigma.
\end{eqnarray}
Notice that on $\partial B_R$, we have the formula
\begin{eqnarray}\label{I1.30}
\eta\cdot\nabla (\eta\cdot\nu)&=&\zeta\langle\nu, \nabla\zeta\rangle
=\zeta {\rm{div}}(\zeta\nu)-\zeta^2{\rm{div}}\nu\nonumber\\
&=& \zeta{\rm{div}} \eta -\zeta^2H.
\end{eqnarray}
Thus we obtain that
\begin{eqnarray}\label{I1.4}
&&\int_{\partial B_R}\frac{1}{2} |\nabla u_0(x)|^2 \eta\cdot\nabla (\eta(x)\cdot\nu(x))\,d\sigma\nonumber\\
&&=\int_{\partial B_R}\frac{1}{2} |\nabla u_0(x)|^2\big(\zeta(x) {\rm{div}} \eta(x) -\zeta^2(x) H(x)\big)\,d\sigma.
\end{eqnarray}

Substituting \eqref{I1.1}, \eqref{I1.2}, \eqref{I1.3}, \eqref{I1.4} into \eqref{I1} and applying \eqref{ling1}, we obtain that
\begin{eqnarray}\label{I1.5}
I'(0)&=&\int_{\partial B_R}\frac{1}{2} |\nabla u_0(x)|^2\zeta(x) {\rm{div}} \eta(x)\,d\sigma\nonumber\\
&=&\frac{R^2}{2n^2}\int_{\partial B_R}\zeta(x) {\rm{div}} \eta(x)\,d\sigma=0.
\end{eqnarray}
%\begin{eqnarray}
% &=& \int_{\partial B_R} (-\frac{x}{n}) \nabla \dot{u_0} \zeta +  (-\frac{\eta }{n}) \cdot (-\frac{x}{n} \eta \cdot \nu)\nonumber\\
% &=& -\frac{R}{n^2}\int_{\partial B_R} \zeta^2+\frac{R}{n^2}\int_{\partial B_R} \zeta^2=0,
%\end{eqnarray}
%, and 
%\begin{align*}
%\frac{d}{dt} \big |_{t=0} (\eta \cdot \nu)=&\nabla \eta \, \eta \cdot \nu, \quad \mbox{since $\eta \cdot \frac{d}{dt} \nu \big |_{t=0}=0$. }  \\
% =& \zeta \left(\div \eta -\div _{\partial B_R} (\zeta \nu) \right)\\
% =& \zeta \div \eta -\zeta^2 H, \end{align*}

Next,  by \eqref{ling1} and \eqref{surface-derivative}, we compute
\begin{eqnarray}\label{III}
    III'(0)&=&-\frac{d}{dt}\big|_{t=0}\int_{\partial \Omega(t)}u(t,x) \eta(x) \cdot \nu(t,x)\,d\sigma\nonumber\\
    &=&- \int_{\partial B_R} (v(x)\zeta(x)+u_0(x)\eta(x)\cdot\frac{\partial\nu}{\partial t}(0,x))\,d\sigma\nonumber\\
    &&-\int_{\partial B_R} \big(\frac{\partial u_0(x)}{\partial\nu}\zeta^2(x) +u_0(x) \eta\cdot\nabla (\eta(x)\cdot\nu(x))\big)\,d\sigma\nonumber\\
    &&-\int_{\partial B_R} u_0(x) \zeta^2(x) H(x)\,d\sigma\nonumber\\
    %&=&-\int_{\partial B_R} (\dot{u_0}+\nabla u_{B_R} \cdot \eta) \eta \cdot \nu +u_0(\zeta \div \eta)\nonumber\\
    &=& -\int_{\partial B_R} \big(v(x)\zeta(x) -\frac{R}{n} \zeta^2(x)\big)\,d\sigma-\int_{\partial B_R} u_0(x) \zeta(x){\rm{div}}\eta(x)\,d\sigma\nonumber\\
    &=& \frac{R}{n}\int_{\partial B_R} \zeta^2(x)\,d\sigma-\int_{\partial B_R} v(x)\zeta(x)\,d\sigma,
\end{eqnarray}
where we have used the fact $\eta(x)\cdot\frac{\partial\nu}{\partial t}(0,x)=0$, \eqref{I1.30}, and \eqref{I1.5} on $\partial B_R$.

Recall that the mean curvature of $\partial\Omega(t)$ satisfies (see Huisken \cite{Huisken})
\begin{equation}\label{IV0}
\frac{\partial H}{\partial t}(t,x)=-\Delta_{\partial \Omega(t)} \eta(x)\cdot\nu(t,x) -|A(t,x)|^2 \eta(x)\cdot\nu(t,x), \ x\in\partial\Omega(t),
\end{equation}
where $\Delta_{\partial\Omega(t)}$ is the Laplace operator on $\partial\Omega(t)$, and $A$ is the second fundamental form of $\partial\Omega(t)$.

Applying \eqref{ling}, \eqref{surface-derivative},  \eqref{I1.30}, and \eqref{I1.4} we can compute
\begin{eqnarray}\label{II}
    II'(0)&=&-\frac{d}{dt}\big|_{t=0} \Big[\big(\frac{1}{m}\int_{\partial \Omega(t)} u(t,x)\,d\sigma\big)^2\int_{\partial \Omega(t)} \eta(x) \cdot \nu(t,x)\,d\sigma \Big]\nonumber\\
    &=&-\big(\frac{1}{m}\int_{\partial B_R} u_0\big)^2 \frac{d}{dt}\big|_{t=0}\int_{\partial \Omega(t)}\eta(x) \cdot \nu(t,x)\,d\sigma\nonumber\\
    &=&-\frac{R^2}{n^2}\int_{\partial B_R}\big(\eta(x)\cdot\frac{\partial\nu}{\partial t}(0,x)+\eta(x)\cdot\nabla(\eta(x)\cdot\nu(x)) +H(x)\zeta^2(x)\big)\,d\sigma\nonumber\\
    &=&-\frac{R^2}{n^2}\int_{\partial B_R}\big(\eta(x)\cdot\frac{\partial\nu}{\partial t}(0,x)+(\zeta(x){\rm{div}}\eta(x)-\zeta^2(x)H(x))+H(x)\zeta^2(x)\big)\,d\sigma\nonumber\\
    &=&-\frac{R^2}{n^2}\int_{\partial B_R}\zeta(x){\rm{div}}\eta(x)\,d\sigma=0.
\end{eqnarray}
Applying \eqref{ling}, \eqref{surface-derivative}, \eqref{IV0}, and \eqref{I1.30}, and using
$$ \frac{1}{m}\int_{\partial B_R} u_0(x)\,d\sigma=\frac{R}{n}, \ |A(x)|^2=\frac{n-1}{R^2} \ {\rm{for}}\ x\in\partial B_R,$$
we can compute
\begin{eqnarray}\label{IV}
    &&IV'(0)\nonumber\\
    &&=\frac{d}{dt}\big|_{t=0}\int_{\partial \Omega(t)} \big(\frac{1}{m}\int_{\partial \Omega(t)} u(t,y)\,d\sigma\big)u(t,x) H(t,x) \eta(x) \cdot \nu(t,x)\,d\sigma\nonumber\\
    &&=\int_{\partial B_R} \frac{d}{dt} \big|_{t=0}\big(\frac{1}{m}\int_{\partial \Omega(t)} u(t,y)\,d\sigma\big)u_0(x) H(x) \zeta(x)\,d\sigma\nonumber\\
    &&\ \ +\big(\frac{1}{m}\int_{\partial B_R} u_0(x)\,d\sigma\big) \int_{\partial B_R}(v(x)+\eta(x)\cdot\nabla u_0(x))H(x) \zeta(x)\,d\sigma\nonumber\\
    &&\ \ + \big(\frac{1}{m}\int_{\partial B_R} u_0(x)\,d\sigma\big)\int_{\partial B_R} u_0(x)(-\Delta_{\partial B_R} \zeta(x) -|A(x)|^2 \zeta(x)) \zeta(x)\,d\sigma\nonumber\\
    &&\ \ +\big(\frac{1}{m}\int_{\partial B_R} u_0(x)\,d\sigma\big)\int_{\partial B_R} u_0(x)H(x)(\zeta(x) {\rm{div}}\eta(x) -\zeta^2(x) H(x))\,d\sigma\nonumber\\
    &&\ \ + \big(\frac{1}{m}\int_{\partial B_R} u_0(x)\,d\sigma\big)\int_{\partial B_R} u_0(x)H^2(x) \zeta^2(x)\,d\sigma\nonumber \\
    &&=\frac{R}{n}\big[\frac{n-1}{R}\int_{\partial B_R}v(x)\zeta(x)\,d\sigma-\frac{n-1}{n}\int_{\partial B_R}\zeta^2(x)\,d\sigma\nonumber\\
    &&\ \ +\frac{m}{n^2\omega_n R^{n-2}}\int_{\partial B_R}(-\Delta_{\partial B_R} \zeta(x) -|A(x)|^2 \zeta(x))\zeta(x)\,d\sigma\big]\nonumber\\
    &&=-\frac{(n-1)R}{n^2}\int_{\partial B_R} \zeta^2(x)\,d\sigma+\frac{n-1}{n}\int_{\partial B_R} v(x)\zeta(x)\,d\sigma\nonumber\\
    &&\ \ +\frac{m}{n^3\omega_n R^{n-3}}\int_{\partial B_R} \big(|\nabla_{\rm{tan}} \zeta(x)|^2-\frac{n-1}{R^2}\zeta^2(x)\big)\,d\sigma.
\end{eqnarray}
Therefore, by adding \eqref{I1}, \eqref{II}, \eqref{III}, and \eqref{IV} together, we obtain
\begin{eqnarray}\label{2nd-variation}
   \frac{d^2}{dt^2}\big|_{t=0}\mathcal{J}_m(u(t), \Omega(t)) &=& I'(0)+II'(0)+III'(0)+IV'(0)=I'(0)+IV'(0)\nonumber\\
    &=& \frac{R}{n^2}\int_{\partial B_R} \zeta^2(x)\,d\sigma-\frac{1}{n}\int_{\partial B_R} v(x) \zeta(x)\,d\sigma\nonumber\\
    && + \frac{m}{n^3\omega_n R^{n-3}}\int_{\partial B_R}\big(|\nabla_{\rm{tan}} \zeta(x)|^2-\frac{n-1}{R^2}\zeta^2(x)\big)\,d\sigma.
\end{eqnarray}
Since $\int_{\partial B_R}\zeta(x)\,d\sigma=0$, it follows from the Poincar\'e inequality on $\partial B_R$ that
\begin{equation}\label{poincare}
\int_{\partial B_R}\big(|\nabla_{\rm{tan}} \zeta(x)|^2-\frac{n-1}{R^2}\zeta^2(x)\big)\,d\sigma\ge 0.
\end{equation}
Now we claim that 
\begin{equation}\label{stekloff}
 \frac{R}{n^2}\int_{\partial B_R} \zeta^2(x)\,d\sigma-\frac{1}{n}\int_{\partial B_R} v(x) \zeta(x)\,d\sigma\ge 0.
\end{equation}
To see this, notice that by \eqref{zero-average}, $\displaystyle\int_{\partial B_R} v(x)\,d\sigma=0$. 
Recall that the first Stekloff eigenvalue on $B_R$ is $\frac{1}{R}$, which implies that
\begin{align}\label{stekloff1}
\int_{\partial B_R} v^2(x)\,d\sigma\le R\int_{B_R}|\nabla v(x)|^2\,dx.
\end{align}
Applying the equation \eqref{eqnvelocity} for $v$, we have
\begin{eqnarray*}
   \int_{B_R}|\nabla v(x)|^2\,dx &=&\int_{\partial B_R} \frac{\partial v(x)}{\partial\nu}v(x)\,d\sigma
    = \frac{1}{n}\int_{\partial B_R} \zeta(x) v(x)\,d\sigma\\
    &\le & \frac{1}{n}\big(\int_{\partial B_R} v^2(x)\,d\sigma \big)^{\frac12} \big(\int_{\partial B_R}\zeta^2(x)\,d\sigma\big)^{\frac12}\\
    &\le & \frac{R^{\frac12}}{n}\big(\int_{B_R}|\nabla v(x)|^2\,d\sigma\big)^{\frac12}
    \big(\int_{\partial B_R}\zeta^2(x) d\sigma\big)^{\frac12}.
\end{eqnarray*}
This implies  
\begin{align}
   \frac{1}{n}\int_{\partial B_R} \zeta(x) v(x)\,d\sigma= \int_{B_R} |\nabla v(x)|^2\,dx\le \frac{R}{n^2}\int_{\partial B_R} \zeta^2(x)\,d\sigma.
\end{align}
Hence \eqref{stekloff} holds. Putting \eqref{poincare} and \eqref{stekloff} into \eqref{2nd-variation}, we conclude that 
 $$\frac{d^2}{dt^2}\big|_{t=0}\mathcal{J}_m(u(t), \Omega(t))\ge 0.$$
This completes the proof. 
\end{proof}

\bigskip
\bigskip
\noindent{\bf Acknowledgements}. Both first and third authors are partially supported by NSF DMS grant 1764417.

\end{document}